\documentclass[12pt, reqno]{article}
\usepackage[utf8]{inputenc}
\usepackage[breaklinks=True]{hyperref}
\usepackage{amsmath, amsfonts, amssymb, amsthm, mathtools, comment, float, todonotes, enumitem, booktabs}
\usepackage[margin=1.5in]{geometry}

\newcommand{\set}[1]{\left\{ #1 \right\}}
\newcommand{\abs}[1]{\left| #1 \right|}
\newcommand{\tup}[1]{\left( #1 \right)}
\newcommand{\ive}[1]{\left[ #1 \right]}

\newcommand{\Fix}{\operatorname{fix}}

\newcommand{\des}{\operatorname{des}}

\newcommand{\asc}{\operatorname{asc}}

\newcommand{\exc}{\operatorname{exc}}
\newcommand{\anxc}{\operatorname{anxc}}
\newcommand{\cyc}{\operatorname{cyc}}
\newcommand{\wexc}{\operatorname{wexc}}
\newcommand{\F}{\mathcal{F}}
\newcommand{\R}{\mathcal{R}}
\newcommand{\C}{\mathcal{C}}
\newcommand{\I}{\mathcal{I}}
\newcommand{\Qsq}{\mathcal{Q}^2}
\newcommand{\D}{\mathcal{D}}
\newcommand{\RtLmax}{\underleftarrow{\max}}
\newcommand{\LtRmax}{\underrightarrow{\max}}
\newcommand{\RtLmin}{\underleftarrow{\min}}

\newcommand{\email}[1]{\href{mailto:#1}{\texttt{#1}}}

\newtheoremstyle{plainsl}% <name>
  {8pt plus 2pt minus 4pt}% <Space above>
  {8pt plus 2pt minus 4pt}% <Space below>
  {\slshape}% <Body font>
  {0pt}% <Indent amount>
  {\bfseries}% <Theorem head font>
  {.}% <Punctuation after theorem head>
  {5pt plus 1pt minus 1pt}% <Space after theorem headi>
  {}% <Theorem head spec (can be left empty, meaning `normal')>

\theoremstyle{plainsl}
  \newtheorem{theorem}{Theorem}[section]
  \newtheorem{proposition}[theorem]{Proposition}
  \newtheorem{lemma}[theorem]{Lemma}
  \newtheorem{corollary}[theorem]{Corollary}
  
\theoremstyle{definition}
  \newtheorem{definition}[theorem]{Definition}
  
\theoremstyle{remark}
  \newtheorem{remark}[theorem]{Remark}

\usepackage{tikz, pgfkeys, subcaption}
\newif\ifpermgridfitlabels
\newif\ifpermgridstaggercols
\newif\ifpermgridstaggerrows
\newsavebox{\permgridlabelbox}
\newlength{\permgridavaillen}

\pgfkeys{
  /permgrid/.is family,
  /permgrid,
  highlight x/.store in=\permgridx,
  highlight y/.store in=\permgridy,
  highlight color/.store in=\permgridcolor,
  highlights/.store in=\permgridhighlights,
  scale/.store in=\permgridscale,
  labelscale/.store in=\permgridlabelscale,
  labelscale min/.store in=\permgridlabelscalemin,
  fit labels/.is if=permgridfitlabels,
  label style/.code={\tikzset{permgrid label style/.style={#1}}},
  column label style/.code={\tikzset{permgrid col style/.style={#1}}},
  row label style/.code={\tikzset{permgrid row style/.style={#1}}},
  stagger columns/.is if=permgridstaggercols,
  stagger rows/.is if=permgridstaggerrows,
  stagger labels/.code={\pgfkeys{/permgrid/stagger columns=#1,
                                 /permgrid/stagger rows=#1}},
  stagger labels/.default=true,
  stagger amount/.store in=\permgridstaggeramount,
}

\newcommand{\permgridquad}[5]{%
  \edef\pgqdir{#3}%
  \def\pgqNE{NE}\def\pgqNW{NW}\def\pgqSE{SE}\def\pgqSW{SW}%
  \def\pgqNEw{NEw}\def\pgqNWw{NWw}\def\pgqSEw{SEw}\def\pgqSWw{SWw}%
  \ifx\pgqdir\pgqNE  \fill[#4] (#1,#2)     rectangle (#5,#5);    \fi
  \ifx\pgqdir\pgqNW  \fill[#4] (0,#2)      rectangle (#1-1,#5);  \fi
  \ifx\pgqdir\pgqSE  \fill[#4] (#1,0)      rectangle (#5,#2-1);  \fi
  \ifx\pgqdir\pgqSW  \fill[#4] (0,0)       rectangle (#1-1,#2-1);\fi
  \ifx\pgqdir\pgqNEw \fill[#4] (#1-1,#2-1) rectangle (#5,#5);    \fi
  \ifx\pgqdir\pgqNWw \fill[#4] (0,#2-1)    rectangle (#1,#5);    \fi
  \ifx\pgqdir\pgqSEw \fill[#4] (#1-1,0)    rectangle (#5,#2);    \fi
  \ifx\pgqdir\pgqSWw \fill[#4] (0,0)       rectangle (#1,#2);    \fi
}

\newcommand{\permgridsetdefaults}{%
  \def\permgridx{}%
  \def\permgridy{}%
  \def\permgridcolor{red!15}%
  \def\permgridscale{0.8}%
  \def\permgridhighlights{}%
  \def\permgridlabelscale{1}%
  \def\permgridlabelscalemin{0.4}%
  \tikzset{permgrid label style/.style={}}%
  \tikzset{permgrid col style/.style={}}%
  \tikzset{permgrid row style/.style={}}%
  \def\permgridstaggeramount{1.5ex}%
  \permgridfitlabelsfalse
  \permgridstaggercolsfalse
  \permgridstaggerrowsfalse
}

\newcommand{\permgrid}[3][]{%
  \begingroup
  \permgridsetdefaults
  \pgfkeys{/permgrid,#1}%
  \begin{tikzpicture}[
    scale=\permgridscale,
    baseline=(current bounding box.center)
  ]
    \ifx\permgridx\empty\else
      \ifx\permgridy\empty\else
        \fill[\permgridcolor] (\permgridx-1,\permgridy-1) rectangle (#2,#2);
      \fi
    \fi
    \ifx\permgridhighlights\empty\else
      \foreach \hx/\hy/\hdir/\hcol in \permgridhighlights {%
        \permgridquad{\hx}{\hy}{\hdir}{\hcol}{#2}%
      }%
    \fi
    \draw[step=1] (0,0) grid (#2,#2);
    \foreach \i in {1,...,#2}
      \node[below] at (\i-0.5,0) {\(\i\)};
    \foreach \j in {1,...,#2}
      \node[left] at (0,\j-0.5) {\(\j\)};
    \foreach \i [count=\j from 1] in {#3}
      \fill (\j-0.5,\i-0.5) circle (2.5pt);
  \end{tikzpicture}%
  \endgroup
}

\newcommand{\permgridcellsize}{%
  \setlength{\permgridavaillen}{\permgridscale cm}%
  \setlength{\permgridavaillen}{0.92\permgridavaillen}%
}

\newcommand{\permgridpoints}[3][]{%
  \begingroup
  \permgridsetdefaults
  \pgfkeys{/permgrid,#1}%
  \global\let\permgridautoscale\permgridlabelscale
  \ifpermgridfitlabels
    \permgridcellsize
    \foreach \px/\py/\pcol/\prow in {#3} {%
      \ifx\pcol\empty\else
        \sbox\permgridlabelbox{\pcol}%
        \pgfmathsetmacro{\permgridtmpscale}{%
          min(\permgridautoscale,
              \the\permgridavaillen / max(\the\wd\permgridlabelbox, 0.1pt))}%
        \global\let\permgridautoscale\permgridtmpscale
      \fi
      \ifx\prow\empty\else
        \sbox\permgridlabelbox{\prow}%
        \pgfmathsetmacro{\permgridtmpscale}{%
          min(\permgridautoscale,
              \the\permgridavaillen
                / max(\the\ht\permgridlabelbox + \the\dp\permgridlabelbox, 0.1pt))}%
        \global\let\permgridautoscale\permgridtmpscale
      \fi
    }%
    \pgfmathsetmacro{\permgridtmpscale}{max(\permgridautoscale,\permgridlabelscalemin)}%
    \global\let\permgridautoscale\permgridtmpscale
  \fi
  \begin{tikzpicture}[
    scale=\permgridscale,
    baseline=(current bounding box.center),
    permgrid label/.style={scale=\permgridautoscale, permgrid label style},
  ]
    \ifx\permgridx\empty\else
      \ifx\permgridy\empty\else
        \fill[\permgridcolor] (\permgridx-1,\permgridy-1) rectangle (#2,#2);
      \fi
    \fi
    \ifx\permgridhighlights\empty\else
      \foreach \hx/\hy/\hdir/\hcol in \permgridhighlights {%
        \permgridquad{\hx}{\hy}{\hdir}{\hcol}{#2}%
      }%
    \fi
    \draw[step=1] (0,0) grid (#2,#2);
    \foreach \px/\py/\pcol/\prow in {#3} {%
      \fill (\px-0.5,\py-0.5) circle (2.5pt);
    }
    \foreach \px/\py/\pcol/\prow [count=\pidx] in {#3} {%
      \def\permgridcolshift{0}\def\permgridrowshift{0}%
      \ifodd\pidx\else
        \ifpermgridstaggercols \def\permgridcolshift{1}\fi
        \ifpermgridstaggerrows \def\permgridrowshift{1}\fi
      \fi
      \ifx\pcol\empty\else
        \node[permgrid label, anchor=north,
              yshift=-\permgridcolshift*\permgridstaggeramount,
              permgrid col style]
          at (\px-0.5,0) {\pcol};
      \fi
      \ifx\prow\empty\else
        \node[permgrid label, anchor=east,
              xshift=-\permgridrowshift*\permgridstaggeramount,
              permgrid row style]
          at (0,\py-0.5) {\prow};
      \fi
    }
  \end{tikzpicture}%
  \endgroup
}

\usepackage{xargs}

\title{Foatic Homomesy}
\author{Peter Mao\footnote{Department of Mathematics, Drexel University,
\email{pm929@drexel.edu}, \url{https://maopeter.github.io/}}}
\date{\today}
\begin{document}

\maketitle

\begin{abstract}
Using a unified method of identifying a telescoping function for each statistic-action pair, we prove conjectures of La Croix and Roby, and of Sheridan-Rossi that various permutation statistics, such as the number of fixed points, are homomesic under Foatic actions: compositions of the Foata bijection and dihedral symmetries of one-line notation. Underlying these results is a family of bijective identities relating permutation statistics across the Foata bijection; one of them --- an unexpected invariance of the number of right-to-left minima under reverse-complement, proved by an explicit bijection --- may be of independent interest. The telescoping functions also turn out to be shared among some Foatic actions, implying that mixing certain actions still yields a constant average on closed walks.
\medskip

\medskip

\end{abstract}

\section{Introduction}

The Foata bijection, also called the R\'enyi--Foata transform \cite{Renyi1962}, relates cycle notation and one-line notation for permutations, and provides a bijective explanation for the equidistribution of the permutation statistics $\operatorname{exc}$ and $\operatorname{des}$ \cite{FoataSchutzenberger1970}.
Being a bijection from a symmetric group $S_n$ to itself, the Foata bijection can be studied as a discrete dynamical system.
From this viewpoint (called ``dynamical algebraic combinatorics'' in \cite{Roby2016}), La Croix and Roby \cite{LaCroixRoby2020FoaticAO} and then Sheridan-Rossi \cite{Sheridan-Rossi2020} analyzed the algebraic properties of its orbits,
discovering certain homomesies for permutation bijections constructed by composing the Foata bijection with simple involutive bijections.
(In the words of Propp and Roby \cite{Propp_Roby_2015}, a \emph{homomesy} occurs when a statistic has a constant average on every orbit of a discrete dynamical system.)
On the computational and experimental side, first Elder et al. \cite{ElderEtAl2024} and then Dowling and Lafreni\`ere \cite{Dowling_2025} have used the FindStat database to discover permutation statistic homomesies. Of particular note to our results is that Elder et al. found no statistic to be homomesic under the Foata bijection itself.

In Section \ref{mainthms}, we state the homomesy conjectures that were left open in \cite{LaCroixRoby2020FoaticAO} and \cite{Sheridan-Rossi2020}, and then we prove them in Section \ref{mainpfs}.
Along the way, we also strengthen the characterization of these homomesy results by showing that some homomesies are compatible with each other.
That is, in some cases, we show that averages equal.
Not just over all orbits of a given bijection, but also over closed walks formed by any mix of several bijections (freely deciding on which bijection to take at each given step).
This follows from showing that some homomesies share telescoping functions which govern their averaging behavior.
Our proofs proceed through bijective identities for permutation statistics under the Foata bijection, one of which may be of independent interest: the number of right-to-left minima of $\F^{-1}(\theta)$ is unchanged compared with $\F^{-1}(\Qsq(\theta))$ where $\Qsq$ is reverse-complementation (Proposition \ref{prop.RtL_foata_rotation}), which we establish with an explicit bijection.

\subsection*{On the use of AI}

Both the research and writing portions of the production of this paper were done in collaboration with various AI models. Writing assistance came in the form of drafting standard definitions, formatting tables, editing, and proofreading. All of what is written has been proofread and attested to by the author.

On the mathematical side, AI was mainly useful on two fronts. First was the identification of certain telescoping functions (see $f$ in Lemma \ref{lem.telescope_strategy}) from example data. As seen in tables \ref{tab.proved_telescopes} and \ref{tab.conjectured_telescopes}, not all of the functions turned out to be straightforward linear combinations of well-known permutation statistics (see the conjectured $f$ for $\R-\Qsq$). Therefore it was highly helpful to allow AI to work out these functions autonomously. The second front where AI was useful was in writing code to check the claims of lemmas and even individual proof steps very quickly. This too was accomplished mostly autonomously by providing the draft .tex file and requesting code to verify or locate mistakes in the proofs given.

\subsection*{Acknowledgements}

We thank Darij Grinberg for suggesting the approach to proving homomesy results and help refining some of the proofs into much more digestible forms.

\section{Definitions}

This section defines the main topics of study, homomesy (Definition \ref{def.homomesy}) and Foatic maps (Definition \ref{def.foaticmap}), as well as any conventions and notations that we will be using.

Originally coined by James Propp and Tom Roby \cite{Propp_Roby_2015}, homomesy refines taking statistic averages from globally over a set of objects $X$ to $\tau$-orbits of $X$ for some invertible map $\tau$.
%This motivates also taking other probability inspired measures (variance, E[X^n])?

\begin{definition}
    \label{def.homomesy}
    Let $X$ be a set.
    
    \begin{enumerate}

    \item A \emph{statistic} on $X$ simply means a map from $X$ to a further set. In our cases, statistics will be numerical, i.e., their values will be numbers.

    \item
    An \emph{action} on $X$ shall mean an invertible map $\tau : X \to X$ whose orbits on $X$ are all finite.
    
    \item We say a statistic $s : X \to \mathbb{C}$ is \emph{homomesic} under a given action $\tau$ on $X$ if there exists some constant $c \in \mathbb{C}$ such that for every $\tau$-orbit $\mathcal{O}$ in $X$, we have
    \begin{align*}
        \frac{1}{|\mathcal{O}|} \sum_{x \in \mathcal{O}} s(x) = c.
    \end{align*}
    In other words, the average value of $s$ over any $\tau$-orbit is always the same constant $c$.
    In this situation, we say that the statistic $s$ under the action of $\tau$ on $X$ is \emph{$c$-mesic}.

    \end{enumerate}
\end{definition}

Of course, if $X$ is finite, then an action on $X$ is the same as a permutation of $X$.
However, most of the time, our set $X$ will itself consist of permutations, so the word ``action'' for a permutation \textbf{of} $X$ will help us distinguish such from permutations \textbf{in} $X$.

We are interested in homomesy of permutation statistics, so we both define some basics which will be familiar and some conventions which we will be using going forward.

\begin{definition}
    Let $n \geq 0$ be an integer. Define the set $[n] := \{1, 2, \dots, n\}$.

    A \emph{permutation} of $[n]$ is a bijection $\sigma : [n] \to [n]$.
    The group formed by all those permutations from $[n]$ to $[n]$ is written as $S_n$.
    The multiplication on $S_n$ is composition of permutations.
\end{definition}

\begin{definition}
    A permutation $\sigma \in S_n$ can be encoded as the word $\sigma_1 \sigma_2 \cdots \sigma_n$, where $\sigma_i = \sigma(i)$ for each $i \in [n]$.
    This is called the \emph{one-line notation} (or, short, \emph{OLN}) of $\sigma$.
    Note that this word $\sigma_1 \sigma_2 \dots \sigma_n$ consists of the numbers from $1$ up to $n$ without duplicates.
\end{definition}

\begin{definition}
    A permutation $\sigma \in S_n$ can be encoded as a list of all its disjoint cycles, including the $1$-cycles (i.e., fixed points). 
    These cycles can be ordered arbitrarily and are only unique up to cyclic rotation.
    If we additionally require that each cycle starts with its largest element, and that the cycles are ordered by their largest elements (i.e., their largest elements increase as we go to right), then this list becomes genuinely unique, and is called the \emph{canonical cycle decomposition} (or, short, \emph{CCD}) of $\sigma$.
    
    %is said to be written as its \emph{canonical cycle decomposition} or \emph{CCD} when written as the product of disjoint cycles where each cycle starts (the leftmost element) with its largest element and the cycles are ordered by their largest elements from smallest to largest going left to right. Those cycles can be read as a map by interpreting two subsequent elements both in the same set of parentheses $i$ and $j$ from left to right as $\sigma(i) = j$. When an element $l$ is the last in a set of parentheses and $m$ is the first element in that set of parentheses, $\sigma(l) = m$.

    For example, if $\sigma = 417682953 \in S_9$ (shown here in one-line notation), then a list of all cycles of $\sigma$ is
    $$(1, 4, 6, 2)(3, 7, 9)(5, 8);$$
    but this is not the CCD, since the largest element in each cycle does not come first.
    Putting the largest elements of each cycle first, the list becomes
    $$(6,2,1,4)(9,3,7)(8,5).$$
    This is still not the CCD, since the cycles are out of order. In increasing order, the largest elements are $6, 8, 9$, so the CCD of $\sigma$ is given by
    $$\sigma = (6,2,1,4)(8,5)(9,3,7).$$
    We shall always delimit the cycles in a CCD by parentheses.
    
    Note that if we interpret each cycle in the CCD of $\sigma$ as a cyclic permutation, then the product of these cyclic permutations is $\sigma$.
\end{definition}

In this work, we shall study the orbits of permutations under certain actions formed by composing (``intertwining'') the Foata bijection with certain simple involutions.
These compositions are the so-called \emph{Foatic maps}, originally introduced by Michael La Croix and Tom Roby \cite{LaCroixRoby2020FoaticAO} and then studied further by Elizabeth Sheridan-Rossi \cite{Sheridan-Rossi2020}.
We shall soon recall their definitions.
First, however, we will define some basic features of permutations.

We begin with the permutation statistics.
Throughout, $\sigma \in S_n$ is written in OLN, and we illustrate each definition with the running example $\pi = 14253 \in S_5$.

\begin{definition}
    Let $\sigma \in S_n$ and $i \in [n]$.
    We call the position $i$ a \emph{fixed point} of $\sigma$ if $\sigma(i) = i$, an \emph{excedance} of $\sigma$ if $\sigma(i) > i$, an \emph{antiexcedance} of $\sigma$ if $\sigma(i) < i$, and a \emph{weak excedance} if $\sigma(i) \geq i$ (equivalently, $i$ is an excedance or a fixed point).
    For $i \in [n-1]$, we call the position $i$ a \emph{descent} of $\sigma$ if $\sigma(i) > \sigma(i+1)$, and an \emph{ascent} of $\sigma$ if $\sigma(i) < \sigma(i+1)$.

    For the running example $\pi = 14253$, position $1$ is the unique fixed point; positions $2$ and $4$ are excedances; positions $3$ and $5$ are antiexcedances; the weak excedances are positions $1$, $2$, and $4$; the descents are positions $2$ and $4$; and the ascents are positions $1$ and $3$.
\end{definition}

\begin{definition}
    Let $\sigma \in S_n$ and $i \in [n]$.
    The entry $\sigma(i)$ is
    \begin{enumerate}
        \item a \emph{left-to-right maximum} of $\sigma$ if $\sigma(i) > \sigma(j)$ for every $j < i$;
        \item a \emph{right-to-left maximum} of $\sigma$ if $\sigma(i) > \sigma(j)$ for every $j > i$; and
        \item a \emph{right-to-left minimum} of $\sigma$ if $\sigma(i) < \sigma(j)$ for every $j > i$.
    \end{enumerate}

    For $\pi = 14253$, the left-to-right maxima are $1$, $4$, and $5$; the right-to-left maxima are $5$ and $3$; and the right-to-left minima are $1$, $2$, and $3$.
\end{definition}

\begin{definition}
    The statistics $\Fix$, $\exc$, $\anxc$, $\wexc$, $\des$, and $\asc$ are defined on $S_n$ and send each permutation $\sigma \in S_n$ to, respectively, the number of fixed points, excedances, antiexcedances, weak excedances, descents, and ascents of $\sigma$.
    Likewise $\LtRmax$, $\RtLmax$, and $\RtLmin$ return the number of left-to-right maxima, right-to-left maxima, and right-to-left minima, and $\cyc$ returns the number of cycles in the CCD of $\sigma$ (counting fixed points).

    For $\pi = 14253$, we have $\Fix(\pi) = 1$, $\exc(\pi) = 2$, $\anxc(\pi) = 2$, $\wexc(\pi) = 3$, $\des(\pi) = 2$, $\asc(\pi) = 2$, $\LtRmax(\pi) = 3$, $\RtLmax(\pi) = 2$, $\RtLmin(\pi) = 3$, and $\cyc(\pi) = 2$.
\end{definition}

\begin{definition}
    For a statement $P$, the \emph{truth value} (aka \emph{Iverson bracket}) $[P]$ equals $1$ if $P$ is true and $0$ otherwise.
    We use it to define ``local'' permutation statistics with values in $\set{0,1}$:
    For $\sigma \in S_n$ and $k \in [n-1]$, the \emph{descent indicator} at position $k$ is
    $$D_k(\sigma) := [\sigma(k) > \sigma(k+1)].$$
    For $\sigma \in S_n$ and $k \in [n]$, the \emph{fixed-point indicator} at position $k$ is
    $$\Fix_k(\sigma) := [\sigma(k) = k].$$
    Our results use only the indicators at the two ends: $D_1$ and $D_{n-1}$, and $\Fix_1$ and $\Fix_n$.

    For $\pi = 14253$, we have $D_1(\pi) = 0$ and $D_{n-1}(\pi) = D_4(\pi) = 1$, while $\Fix_1(\pi) = 1$ and $\Fix_n(\pi) = \Fix_5(\pi) = 0$.
\end{definition}

Next we introduce the most crucial ingredient of our bijections:

\begin{definition}
    The \emph{Foata bijection} is the map $\F: S_n \to S_n$ that takes a permutation in CCD, drops the parentheses, and interprets the resulting sequence as an OLN.
    
    For example, $\F((6,2,1,4)(8,5)(9,3,7)) = 621485937$.\newline

    The inverse $\F^{-1}: S_n \to S_n$ takes a permutation in OLN and inserts parentheses in appropriate locations to recover a permutation in CCD.
    Specifically, the parentheses are inserted to the left of each left-to-right maximum in the OLN.
    
    For example, $\F^{-1}(621485937) = (6,2,1,4)(8,5)(9,3,7)$ since the left-to-right maxima are $6$, $8$, and $9$.
    
    We refer to \cite[Lemma 3.39]{Bona-CP3} for more about $\F$ and a proof of the above construction of $\F^{-1}$.
    See also \cite[Proposition 1.3.1]{EC1} where $\F$ is called the \emph{fundamental bijection}.\newline
\end{definition}

The Foata bijection is then intertwined with the following five \emph{dihedral involutions} on $S_n$.

\begin{definition}

    The \emph{complement} map $\C: S_n \to S_n$ takes a permutation in OLN and replaces each element $i$ with $n+1-i$.
    
    For example, $\C(621485937) = 489625173$ since $n=9$ and we therefore replace each element $i$ with $10 - i$.\newline

    The \emph{reversal} map $\R: S_n \to S_n$ takes a permutation in OLN and reverses the order of the elements.
    
    For example, $\R(621485937) = 739584126$.\newline

    The \emph{inverse} map $\I: S_n \to S_n$ takes a permutation and produces its inverse.\newline
    
    The \emph{rotation} map $\Qsq: S_n \to S_n$
    (also known as \emph{reverse complementation}) % in genetics but also e.g. in Sagan's book
    % (also known as the \emph{Sch\"utzenberger involution})?
    is the composition $\Qsq = \C \circ \R = \R \circ \C$. This means we can easily calculate the rotation map by reading the complement of the OLN backwards.\newline

    The \emph{rotated inverse} map $\D: S_n \to S_n$ is the composition $\D = \Qsq \circ \I = \I \circ \Qsq$.
    
    The five maps we just defined are called the \emph{dihedral involutions}.
\end{definition}

\begin{definition}
    \label{def.foaticmap}
    The \emph{Foatic maps} (as defined in \cite{LaCroixRoby2020FoaticAO, Sheridan-Rossi2020}) are the compositions of the form
    $$S_n \xrightarrow{\F} S_n \xrightarrow{\mathcal{A}} S_n \xrightarrow{\F^{-1}} S_n \xrightarrow{\mathcal{B}} S_n$$
    where $\mathcal{A}$ and $\mathcal{B}$ are dihedral involutions.

    The Foatic maps will be named after the dihedral involutions $\mathcal{A}$ and $\mathcal{B}$ involved: The above composition will be referred to as ``$\mathcal{A}$-$\mathcal{B}$''.
    For example, $\D \circ \F^{-1} \circ \C \circ \F$ is called the complement-rotated inverse map.
\end{definition}

\section{Main theorems: Foatic Homomesies}
\label{mainthms}

Our main results prove the conjectures made by La Croix and Roby in \cite{LaCroixRoby2020FoaticAO} and by Sheridan-Rossi in \cite{Sheridan-Rossi2020}.
Theorem \ref{thm.mainFix} unifies the result that $\Fix$ is $1$-mesic under reversal-inversion in \cite[Theorem 14]{LaCroixRoby2020FoaticAO} with a proof that the actions of complement-inversion, complement-rotation, and reversal-rotation also have $\Fix$ as $1$-mesic as in \cite[Conjectures 16, 17.1, 18]{LaCroixRoby2020FoaticAO}.
Then Theorem \ref{thm.mainRtL}, Theorem \ref{thm.mainD}, and Theorem \ref{thm.mainFixEnds} prove \cite[Conjecture 2.5.4, Conjecture 2.5.5, and Conjecture 2.5.6]{Sheridan-Rossi2020} respectively.

\begin{theorem}
    \label{thm.mainFix}
    For $n \geq 1$, the statistic $\Fix$ is $1$-mesic under the actions of
    \begin{enumerate}
        \item complement-inversion ($\I \circ \F^{-1} \circ \C \circ \F$),
        \item complement-rotation ($\Qsq \circ \F^{-1} \circ \C \circ \F$),
        \item reversal-rotation ($\Qsq \circ \F^{-1} \circ \R \circ \F$), and
        \item reversal-inversion ($\I \circ \F^{-1} \circ \R \circ \F$).
    \end{enumerate}

    \hyperlink{pf.fix_full}{\emph{Jump to proof.}}
\end{theorem}

\begin{theorem}
    \label{thm.mainRtL}
    For $n \geq 1$, the statistic $\RtLmax - \RtLmin$ is $0$-mesic under the action of rotation-complement ($\C \circ \F^{-1} \circ \Qsq \circ \F$).

    \medskip
    \hyperlink{pf.rtl}{\emph{Jump to proof.}}
\end{theorem}

\begin{theorem}
    \label{thm.mainD}
    For $n \geq 2$, the statistic $D_1 - D_{n-1}$ is $0$-mesic under the action of rotation-reversal ($\R \circ \F^{-1} \circ \Qsq \circ \F$).

    \medskip
    \hyperlink{pf.des_ends}{\emph{Jump to proof.}}
\end{theorem}

\begin{theorem}
    \label{thm.mainFixEnds}
    For $n \geq 1$, the statistic $\Fix_1 - \Fix_n$ is $0$-mesic under the actions of
    \begin{enumerate}
        \item reversal-rotated inverse ($\D \circ \F^{-1} \circ \R \circ \F$),
        \item rotated inverse-reversal ($\R \circ \F^{-1} \circ \D \circ \F$),
        \item complement-rotated inverse ($\D \circ \F^{-1} \circ \C \circ \F$), and
        \item rotated inverse-complement ($\C \circ \F^{-1} \circ \D \circ \F$).
    \end{enumerate}

    \hyperlink{pf.fix_ends}{\emph{Jump to proof.}}
\end{theorem}

\section{The telescope method}

Our proofs of the above claims are based on a common set of basic methods set out by the following lemmas. Based on the coboundary condition given by Propp in \cite[Proposition 2.4]{Propp2021spectral} for $0$-mesic statistics, we give a simple generalization to handle the $c$-mesic situation with no conditions on our function $f$ (Propp uses $g$ for the same purpose). A similar technique is employed by Defant, Hopkins, Poznanovi\'c, and Propp in \cite{DefantHopkinsPoznanovicPropp2023} to prove homomesy results where they note the use of a coboundary condition which does not exactly match the one given by Propp. We formalize the way this comes about in Remark \ref{rmk.telescope_fn}.

\begin{lemma}
    \label{lem.telescope_strategy}
    Let $X$ be a set, $s: X \to \mathbb{C}$ a statistic, $c \in \mathbb{C}$ a constant, and $\tau: X \to X$ an action on $X$.
    If there exists a function $f : X \to \mathbb{C}$ satisfying
    \begin{align}
        s(x) &= f(x) - f(\tau(x)) + c
        \label{eq.telescope}
    \end{align}
    for all $x \in X$, then the statistic $s$ is $c$-mesic under the action of $\tau$.
\end{lemma}
\begin{proof}
    For every $\tau$-orbit $\mathcal{O}$, we have
    \begin{align*}
        \frac{1}{|\mathcal{O}|}\sum_{x \in \mathcal{O}} s(x) &= \frac{1}{|\mathcal{O}|}\sum_{x \in \mathcal{O}} (f(x) - f(\tau(x)) + c) \\
        &= \frac{1}{|\mathcal{O}|}\left(|\mathcal{O}|c + \sum_{x \in \mathcal{O}} f(x) - \sum_{x \in \mathcal{O}} f(\tau(x))\right) \\
        &= c + \frac{1}{|\mathcal{O}|}\left(\sum_{x \in \mathcal{O}} f(x) - \sum_{\tau^{-1}(x) \in \mathcal{O}} f(\tau(\tau^{-1}(x)))\right) \\
        &= c + \frac{1}{|\mathcal{O}|}\left(\sum_{x \in \mathcal{O}} f(x) - \sum_{x \in \mathcal{O}} f(x)\right) \\
        &= c .
        \qedhere
    \end{align*}
\end{proof}

\begin{remark}
    \label{rmk.telescope_fn}
    For a given statistic $s$ and an action $\tau$ where $s$ is $c$-mesic under the action of $\tau$, the function $f$ used in Lemma \ref{lem.telescope_strategy} both is guaranteed to exist and is unique up to an offsetting constant for each orbit.
\end{remark}
\begin{proof}
    To see the existence, let $\Omega$ be the set of orbits under the action of $\tau$.
    For a given orbit
    \[
    \mathcal{O} = \tup{x_1 \overset{\tau}{\mapsto} x_2 \overset{\tau}{\mapsto} \cdots \overset{\tau}{\mapsto} x_{k-1} \overset{\tau}{\mapsto} x_k \overset{\tau}{\mapsto} x_1},
    \]
    the telescoping condition \eqref{eq.telescope} holds for all $x$ in this orbit if and only if the $k$ equations
    \begin{align*}
        f(x_1) - f(x_2) &= s(x_1) - c, \\
        f(x_2) - f(x_3) &= s(x_2) - c, \\
        &\vdots \\
        f(x_{k-1}) - f(x_k) &= s(x_{k-1}) - c, \\
        f(x_k) - f(x_1) &= s(x_k) - c
    \end{align*}
    hold. These form a system of $k$ equations in the $k$ variables $f(x_1), f(x_2), \ldots, f(x_k)$.
    These $k$ equations sum to $0 = 0$, since their left-hand sides sum to $0$ by telescoping and their right-hand sides sum to
    \begin{align*}
        \sum_{i=1}^k (s(x_i) - c)
        = \sum_{i=1}^k s(x_i) - kc
        = \sum_{x \in \mathcal{O}} s(x) - kc
        = 0
    \end{align*}
    as well (because the definition of $c$-mesy yields $\dfrac{1}{\abs{\mathcal{O}}} \sum_{x \in \mathcal{O}} s(x) = c$, so that $\sum_{x \in \mathcal{O}} s(x) = \abs{\mathcal{O}} c = kc$).
    Thus, in order to solve this system, we only need to solve its first $k-1$ equations, which we can do recursively, starting with an arbitrary choice of $f(x_1)$ and then solving the first equation for $f(x_2)$, the second for $f(x_3)$, and so on.
    The last equation will then automatically hold, since the sum of all $k$ equations is $0=0$.

    So we can always produce our desired function $f$.
    Moreover, in constructing the values $f(x_1), f(x_2), \ldots, f(x_k)$, we had a free choice of $f(x_1)$, which affects the remaining values by an additive offset.
    That is, for any given solution for $f(x_1), \dots, f(x_k)$, for any $a \in \mathbb{C}$, we have another solution
    \begin{align*}
        g(x_1) &= f(x_1) + a, \\
        g(x_2) &= f(x_2) + a, \\
        &\vdots \\
        g(x_k) &= f(x_k) + a,
    \end{align*}
    and there are no other free choices (since our $k$ equations uniquely determine each $f(x_i)$ for any given $f(x_1)$).
    Hence, up to a choice of offset for each orbit, $f$ exists uniquely.
\end{proof}

\begin{definition}
    Let $X$ be a set, $s: X \to \mathbb{C}$ a statistic. For a set of actions on $X$ paired with constants $\{(\tau_i, c_i) : i \in [m]\}$, we say those actions $\tau_1, \dots, \tau_m$ are \emph{$s$-compatible} via $f: X \to \mathbb{C}$ if
    \begin{align*}
        s(x) &= f(x) - f(\tau_i(x)) + c_i
    \end{align*}
    for all $x \in X$ and all $i \in [m]$.

    It follows immediately from Lemma \ref{lem.telescope_strategy} that $s$ is $c_i$-mesic under the action of $\tau_i$ for all $i \in [m]$.
\end{definition}

\begin{corollary}
    \label{cor.compat}
    For a set $X$ with $s$-compatible actions $\tau_1, \dots, \tau_m$ via $f$, define the labeled directed multigraph $G$ with vertices $X$ and edges labeled $\tau_i$ connecting each vertex $x$ to $\tau_i(x)$. For any finite closed walk on $G$ of length $l > 0$ traversing $a_i$ edges labeled $\tau_i$, the average value taken by $s$ over the vertices visited by the walk (with multiplicity if a vertex is visited repeatedly) is given by
    \begin{align*}
        \frac{\sum_{i=1}^m a_i c_i}{\sum_{j=1}^m a_j}.
    \end{align*}
\end{corollary}
\begin{proof}
    Let such a closed walk alternating between the vertices visited and the edges taken be given by $x_1, \tau_{i_1}, x_2, \tau_{i_2}, \dots, \tau_{i_l}, x_{l+1}=x_1$ where $l = \sum_{j=1}^{m} a_j$. Then
    \begin{align*}
        \frac{1}{l}\sum_{j=1}^{l} s(x_j) &= \frac{1}{l}\sum_{j=1}^{l} (f(x_j) - f(\tau_{i_j}(x_{j})) + c_{i_j}) \\
        &= \frac{1}{l}\sum_{j=1}^l (c_{i_j} + f(x_j) - f(x_{j+1})) \\
        &= \frac{1}{l}\sum_{j=1}^l c_{i_j} \\
        &= \frac{1}{l}\sum_{i=1}^m a_i c_i \\
        &= \frac{\sum_{i=1}^m a_i c_i}{\sum_{j=1}^m a_j} .
    \end{align*}
\end{proof}

\begin{remark}
    It has previously been remarked by Elder et al. \cite[Lemma 3.11]{ElderEtAl2024} that for a given set $X$, two statistics, $s_1,s_2: X \to \mathbb{C}$, and an action $\tau: X \to X$, if $s_1$ is $c_1$-mesic under $\tau$ and $s_2$ is $c_2$-mesic under $\tau$, then linear combinations of $s_1$ and $s_2$ are also homomesic. In particular, for any $a,b \in \mathbb{C}$, $as_1 + bs_2$ is $(ac_1 + bc_2)$-mesic under $\tau$.
\end{remark}

\section{Proofs}
\label{mainpfs}

We now come to the proofs of Theorem \ref{thm.mainFix}, Theorem \ref{thm.mainRtL}, Theorem \ref{thm.mainD}, and Theorem \ref{thm.mainFixEnds}.

\subsection{Proof of Theorem \ref{thm.mainFix}}

\begin{proof}[\hypertarget{pf.fix_full}{\textbf{Proof of Theorem \ref{thm.mainFix}}}]
    To prove that $\Fix$ is $1$-mesic under the action of complement-inversion, complement-rotation, reversal-rotation, and reversal-inversion, by Lemma \ref{lem.telescope_strategy}, it will be sufficient to prove that for all $\pi \in S_n$ and $\tau$ as any of the aforementioned actions,
    \begin{align}
        \Fix(\pi) &= -\exc(\pi) + \exc(\tau(\pi)) + 1 \label{eq.fix_exc_id}.
    \end{align}

    We will approach this by decomposing each of our actions into their constituent maps and follow how those maps affect what is being counted. That is, our actions decompose as follows:
    \begin{align*}
        \text{complement-inversion: } &\I \circ \F^{-1} \circ \C \circ \F, \\
        \text{complement-rotation: } &\Qsq \circ \F^{-1} \circ \C \circ \F, \\
        \text{reversal-rotation: } &\Qsq \circ \F^{-1} \circ \R \circ \F, \\
        \text{reversal-inversion: } &\I \circ \F^{-1} \circ \R \circ \F.
    \end{align*}

    Then by the lemmas that will follow, we compute for complement-inversion that
    \begin{align*}
        &-\exc(\pi) + \exc(\I \circ \F^{-1} \circ \C \circ \F(\pi)) + 1 \\
        &= -\exc(\pi) + \anxc(\I \circ \I \circ \F^{-1} \circ \C \circ \F(\pi)) + 1 & \text{by Lemma } \ref{lem.exc_inverse_rotation} \\
        &= -\exc(\pi) + \anxc(\F^{-1} \circ \C \circ \F(\pi)) + 1 \\
        &= -\exc(\pi) + \des(\F \circ \F^{-1} \circ \C \circ \F(\pi)) + 1 &\text{by Lemma } \ref{lem.anxc_foata}\\
        &= -\exc(\pi) + \des(\C \circ \F(\pi)) + 1 \\
        &= -\exc(\pi) + \asc(\C \circ \C \circ \F(\pi)) + 1 &\text{by Lemma } \ref{lem.des_complement_reversal} \\
        &= -\exc(\pi) + \asc(\F(\pi)) + 1 \\
        &= -\exc(\pi) + \Fix(\pi) + \exc(\pi) - 1 + 1 &\text{by Lemma } \ref{lem.asc_foatainv} \\
        &= \Fix(\pi) .
    \end{align*}

    The other actions are verified by similar calculations (replacing the constituent maps of the actions and invoking the appropriate half of the lemma), so we now prove the lemmas these calculations rely on.
\end{proof}

\begin{lemma}
    \label{lem.exc_inverse_rotation}
    For every $\pi \in S_n$, $\exc(\pi) = \anxc(\I(\pi)) = \anxc(\Qsq(\pi))$.
\end{lemma}
\begin{proof}
    For every excedance $i$ of $\pi$, by definition $\pi(i) > i$. Thus, $\pi^{-1}(\pi(i)) = i < \pi(i)$, and therefore $\pi(i)$ is an antiexcedance of $\pi^{-1}$.
    
    Conversely, for every antiexcedance $j$ of $\pi^{-1}$, by definition $\pi^{-1}(j) < j = \pi(\pi^{-1}(j))$, which shows that $\pi^{-1}(j)$ is an excedance of $\pi$.
    
    Hence the excedances of $\pi$ are in bijection with the antiexcedances of $\pi^{-1}$; so we can conclude that
    \[
    \exc(\pi) = \anxc(\I(\pi)).
    \]

    A number $i \in [n]$ is an excedance of $\pi$ if and only if $n+1-i$ is an antiexcedance of $\Qsq(\pi)$. This is because the definition of $\Qsq$ shows that $\Qsq(\pi)(n+1-i) = n+1-\pi(i)$, which is $< n+1-i$ if and only if $\pi(i)>i$.
    %    once again we note that $\pi(i) > i$. In OLN, this translates to having $\pi(i)$ in the $i$th position. If we consider the OLN of $\C(\pi)$, then the $i$th position now has $n - \pi(i) + 1$. Therefore the OLN of $\R(\C(\pi)) = \Qsq(\pi)$ has $n - \pi(i) + 1$ in the $(n - i + 1)$th position. If $\pi(i) > i$, then $n - \pi(i) + 1 < n - i + 1$, so $\Qsq(\pi)$ has an antiexcedance in position $n - i + 1$.
    %For every antiexcedance $j$ of $\pi$, we have $\pi(j) < j$. In OLN, this means that $\pi(j)$ is in the $j$th position. The OLN of $\C(\pi)$ then has $n - \pi(j) + 1$ in the $j$th position and the OLN of $\R(\C(\pi)) = \Qsq(\pi)$ has $n - \pi(j) + 1$ in the $(n - j + 1)$th position. Since $\pi(j) < j$, it must be that $n - \pi(j) + 1 > n - j + 1$, so $\Qsq(\pi)$ has an excedance in position $n - j + 1$.

    Thus the excedances of $\pi$ are also in bijection with the antiexcedances of $\Qsq(\pi)$; so we have shown that
    \[
    \exc(\pi) = \anxc(\Qsq(\pi)). \qedhere
    \]
\end{proof}

\begin{lemma}
    \label{lem.anxc_foata}
    For every $\pi \in S_n$, $\anxc(\pi) = \des(\F(\pi))$.
\end{lemma}
\begin{proof}
    In the CCD of $\pi$, each antiexcedance $i$ corresponds to a shape of either $(\dots, i, \pi(i), \dots)$ where $i$ and $\pi(i)$ are in the main body of the cycle, or $(\pi(i), \dots, i)$ where $i$ and $\pi(i)$ are at the ends of the cycle.
    The second case is impossible in CCD, since we know that $\pi(i) < i$ while by definition of the CCD each cycle must start with its largest element, a contradiction.
    Hence, each antiexcedance $i$ must appear in a $(\dots, i, \pi(i), \dots)$ shape in the CCD of $\pi$.
    Hence, when we remove the parentheses and read the resulting string as OLN (applying $\F$), each antiexcedance becomes a descent of $\F(\pi)$.

    Conversely, any descent of $\F(\pi)$ comes from an antiexcedance of $\pi$. Indeed, a descent of $\F(\pi)$ must be a $j \in \ive{n-1}$ such that $\F(\pi)(j) > \F(\pi)(j+1)$.
    If $\F(\pi)(j)$ and $\F(\pi)(j+1)$ belong to the same cycle in the CCD of $\pi$, then $\F(\pi)(j+1)$ follows $\F(\pi)(j)$ on this cycle, and thus $\F(\pi)(j) > \F(\pi)(j+1)$ shows that $\F(\pi)(j)$ is an antiexcedance of $\pi$.
    If they don't, then $\F(\pi)(j)$ must be the last element of a cycle, while $\F(\pi)(j+1)$ starts the next cycle; hence, $\F(\pi)(j+1)$ must be a left-to-right maximum of $\F(\pi)$ (by the construction of $\F^{-1}$). But this is impossible, since $\F(\pi)(j) > \F(\pi)(j+1)$.
    So they do belong to the same cycle, and thus come from an antiexcedance of $\pi$.
    
    Thus we have a bijection between the antiexcedances of $\pi$ and the descents of $\F(\pi)$, implying that
    \[
    \anxc(\pi) = \des(\F(\pi)). \qedhere
    \]
\end{proof}

\begin{lemma}
    \label{lem.des_complement_reversal}
    For every $\pi \in S_n$, $\des(\pi) = \asc(\C(\pi)) = \asc(\R(\pi))$.
\end{lemma}
\begin{proof}
    For every descent $i$ of $\pi$, by definition $\pi(i) > \pi(i+1)$. Taking the complement gives $\C(\pi)(i) = n + 1 - \pi(i)$ and $\C(\pi)(i+1) = n + 1 - \pi(i+1)$. Thus $\C(\pi)(i) < \C(\pi)(i+1)$, so we have an ascent in $\C(\pi)$.

    For every ascent $j$ of $\C(\pi)$, by definition $\C(\pi)(j) < \C(\pi)(j+1)$. Taking the complement gives us $\pi(j) = n + 1 - \C(\pi)(j)$ and $\pi(j+1) = n + 1 - \C(\pi)(j+1)$. Thus $\pi(j) > \pi(j+1)$, so we have a descent in $\pi$. Hence we have a bijection between the descents of $\pi$ and the ascents of $\C(\pi)$, implying that
    \[
    \des(\pi) = \asc(\C(\pi)).
    \]

    For every descent $i$ of $\pi$, once again $\pi(i) > \pi(i+1)$. Then taking the reversal, $\R(\pi)(n-i+1) > \R(\pi)(n-(i+1)+1) = \R(\pi)(n-i)$, which is an ascent in $\R(\pi)$.
    
    For every ascent $j$ of $\pi$, we have $\pi(j) < \pi(j+1)$. Taking the reversal gives us $\R(\pi)(n-j+1) < \R(\pi)(n-(j+1)+1) = \R(\pi)(n-j)$, which is a descent in $\R(\pi)$. Hence the ascents and descents are in bijection, so
    \[
    \des(\pi) = \asc(\R(\pi)). \qedhere
    \]
\end{proof}

\begin{lemma}
    \label{lem.asc_foatainv}
    For every $\pi \in S_n$, we have $\asc(\F(\pi)) = \Fix(\pi) + \exc(\pi) - 1$.
\end{lemma}

\begin{proof}
    Each $i \in \ive{n}$ is an antiexcedance, a fixed point or an excedance of $\pi$; and these three options are mutually exclusive.
    Hence, $n = \anxc(\pi) + \Fix(\pi) + \exc(\pi)$.

    The ascents of $\F(\pi)$ are the elements of $\ive{n-1}$ that are not descents of $\F(\pi)$. Thus,
    \begin{align*}
        \asc(\F(\pi))
        &= n-1-\des(\F(\pi)) & \\
        & = n-1-\anxc(\pi)
        & \text{by Lemma~\ref{lem.anxc_foata}} \\
        &= (\anxc(\pi) + \Fix(\pi) + \exc(\pi))-1-\anxc(\pi) \\
        &\qquad\qquad\qquad \text{using }n = \anxc(\pi) + \Fix(\pi) + \exc(\pi) \\
        &= \Fix(\pi) + \exc(\pi) - 1
    \end{align*}
\end{proof}

\begin{corollary}
    Complement-inversion, complement-rotation, reversal-rotation, and reversal-inversion are all $\Fix$-compatible via
\begin{align*}
    f(\pi) &= -\exc(\pi).
\end{align*}
\end{corollary}
\begin{proof}
    This result immediately follows our proof of Theorem~\ref{thm.mainFix} which shows Equation~\ref{eq.fix_exc_id}.
\end{proof}

\subsection{Proof of Theorem \ref{thm.mainRtL}}

\begin{proof}[\hypertarget{pf.rtl}{\textbf{Proof of Theorem \ref{thm.mainRtL}}}]
    To prove that $\RtLmax - \RtLmin$ is $0$-mesic under the action of rotation-complement, we use Lemma \ref{lem.telescope_strategy} with $f$ defined by $f(\pi) = \RtLmax(\pi)$. Thus, we must show that for all $\pi \in S_n$, we have
    \begin{align*}
        \RtLmax(\pi) - \RtLmin(\pi) &= \RtLmax(\pi) - \RtLmax(\C \circ \F^{-1} \circ \Qsq \circ \F(\pi)) .
    \end{align*}

    Once again we'll first show the computation:
    \begin{align*}
        \RtLmax(\pi) &- \RtLmax(\C \circ \F^{-1} \circ \Qsq \circ \F(\pi)) \\
        &= \RtLmax(\pi) - \RtLmin(\C \circ \C \circ \F^{-1} \circ \Qsq \circ \F(\pi)) &\text{by Lemma \ref{lem.RtL_complement}} \\
        &= \RtLmax(\pi) - \RtLmin(\F^{-1} \circ \Qsq \circ \F(\pi)) \\
        &= \RtLmax(\pi) - \RtLmin(\F^{-1}( \Qsq( \F(\pi)))) \\
        &= \RtLmax(\pi) - \RtLmin(\pi),
    \end{align*}
    since Proposition~\ref{prop.RtL_foata_rotation} (for $\theta = \F(\pi)$) yields
    \[
    \RtLmin(\F^{-1}( \Qsq( \F(\pi)))) = \RtLmin(\F^{-1}(\F(\pi))) = \RtLmin(\pi).
    \]

    Then we prove the supporting lemmas which the computation relied on.
\end{proof}

\begin{lemma}
    \label{lem.RtL_complement}
    For every $\pi \in S_n$, $\RtLmax(\pi) = \RtLmin(\C(\pi))$.
\end{lemma}
\begin{proof}
    For every right-to-left maximum $\pi(i)$ of $\pi$, by definition we know that $\pi(j) < \pi(i)$ for all $j > i$. Taking the complement of $\pi$ replaces $\pi(i)$ by $n - \pi(i) + 1$ and replaces each $\pi(j)$ by $n - \pi(j) + 1$, so the inequality becomes $\C(\pi)(j) > \C(\pi)(i)$ for all $j > i$,
    and therefore $\C(\pi)(i)$ is a right-to-left minimum of $\C(\pi)$.

    The same argument in reverse shows that if $\C(\pi)(i)$ is a right-to-left minimum of $\C(\pi)$, then $\pi(i)$ is a right-to-left maximum of $\pi$.

    Thus we have a bijection between the right-to-left maxima of $\pi$ and the right-to-left minima of $\C(\pi)$.
    Therefore,
    \[
    \RtLmax(\pi) = \RtLmin(\C(\pi)).
    \qedhere
    \]
\end{proof}

% \begin{lemma}
%     \label{lem.exc_cell_no_rtlmin}
%     For every $\pi \in S_n$, if $k < \pi(k)$, then $\pi(k)$ is not a right-to-left minimum of $\pi$.
% \end{lemma}
% \begin{proof}
%     In such a case, there are $n - \pi(k)$ elements greater than $\pi(k)$ and there are $n - k$ positions to the right of $\pi(k)$ in the OLN of $\pi$. By the pigeonhole principle, since $n - k > n - \pi(k)$, there must exist some position to the right of $\pi(k)$ filled by an element less than $\pi(k)$. Thus $\pi(k)$ is not a right-to-left minimum of $\pi$.
% \end{proof}

\begin{lemma}
    \label{lem.exc_cell_no_rtlmin_converse}
    Let $\pi \in S_n$. If $i$ is a right-to-left minimum of $\pi$, then $\pi^{-1}(i) \geq i$.
\end{lemma}

\begin{proof}
    Let $i = \pi(j)$ be a right-to-left minimum of $\pi$.
    Then, the elements $\pi(j+1), \pi(j+2), \ldots, \pi(n)$ are all larger than $\pi(j) = i$, thus belong to the set $\set{i+1, i+2, \ldots, n}$.
    Hence, they are $n-j$ distinct elements in the $\tup{n-i}$-element set $\set{i+1, i+2, \ldots, n}$.
    This forces $n-j \leq n-i$ by the pigeonhole principle.
    Hence, $j \geq i$. That is, $\pi^{-1}(i) \geq i$, since $j = \pi^{-1}(i)$.
\end{proof}

\begin{proposition}
    \label{prop.RtL_foata_rotation}
    Let $\theta \in S_n$. Then,
    \[
    \RtLmin(\F^{-1}(\Qsq(\theta)))
    = \RtLmin(\F^{-1}(\theta)).
    \]
\end{proposition}

%GEOMETRIC APPROACH
\begin{proof}
    Let $\sigma = \F^{-1}(\theta)$ and $\sigma' = \F^{-1}(\Qsq(\theta))$.
    Then we must show that $\RtLmin(\sigma') = \RtLmin(\sigma)$.

    We define the injective maps $\psi_\sigma, \psi_{\sigma'} : [n] \to [n] \times [n]$ by
    \begin{align*}
        \psi_\sigma(i) &:= (i, \sigma(i)) \\
        \psi_{\sigma'}(i) &:= (i, \sigma'(i)).
    \end{align*}
    We use this to define a \emph{diagram} of a permutation $\sigma$ to be the set of points $\{\psi_\sigma(1), \dots, \psi_\sigma(n)\}$ and displayed as points in an $n \times n$ grid.
    For example, for $\sigma = 35142$ we have the diagram:
    
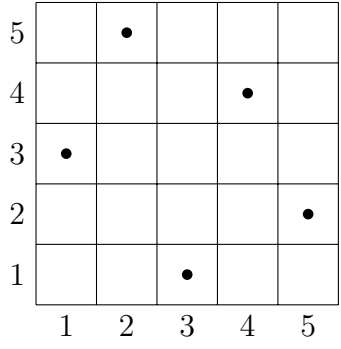
\begin{figure}[H]
        \centering
        \permgrid{5}{3,5,1,4,2}
    \end{figure}
    We will refer to the points by their label along the $x$-axis.
    
    Let us then define the right-to-left minima positions of $\sigma$ and $\sigma'$ by
    \begin{align*}
        M_\sigma &:= \{i \in [n] : \sigma(i) < \sigma(j) \text{ for all } j > i\}, \\
        M_{\sigma'} &:= \{i \in [n] : \sigma'(i) < \sigma'(j) \text{ for all } j > i\}.
    \end{align*}
    These correspond to the points for which there are no other points to their south-east. Thus $2$ and $4$ are not right-to-left minima positions whereas $3$ is:
    
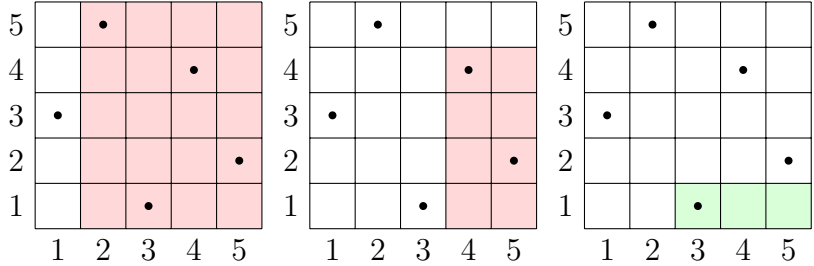
\begin{figure}[H]
        \centering
        \permgrid[scale=0.6, highlights={2/5/SEw/red!15}]{5}{3,5,1,4,2}
        \permgrid[scale=0.6, highlights={4/4/SEw/red!15}]{5}{3,5,1,4,2}
        \permgrid[scale=0.6, highlights={3/1/SEw/green!15}]{5}{3,5,1,4,2}
    \end{figure}
    Their respective sizes are $|M_\sigma| = \RtLmin(\sigma)$ and $|M_{\sigma'}| = \RtLmin(\sigma')$. We also define the involution map $\rho: [n] \times [n] \to [n] \times [n]$ which encodes reflecting the diagram across the line $x+y=n+1$:
    \begin{align*}
        \rho(a,b) &:= (n+1-b, n+1-a).
    \end{align*}
    Notably, this is the same as applying $\D$, the rotated inverse map, to the permutation.
    For our example:
    
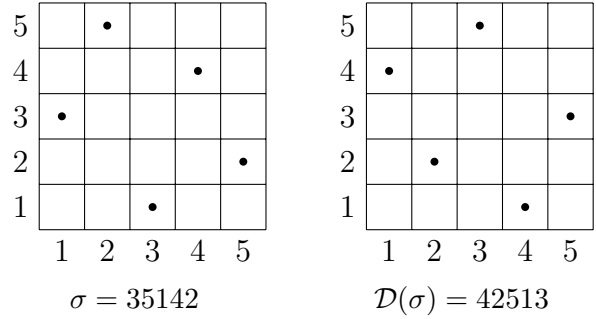
\begin{figure}[H]
        \centering
        \begin{subfigure}[t]{0.3\textwidth}
            \centering
            \permgrid[scale=0.6]{5}{3,5,1,4,2}
            \caption*{$\sigma = 35142$}
        \end{subfigure}
        \begin{subfigure}[t]{0.3\textwidth}
            \centering
            \permgrid[scale=0.6]{5}{4,2,5,1,3}
            \caption*{$\D(\sigma) = 42513$}
        \end{subfigure}
    \end{figure}
    Our claim is then that $\rho$ restricts to a bijection
    \begin{align*}
        \rho_M: \psi_\sigma(M_{\sigma}) \to \psi_{\sigma'}(M_{\sigma'});
    \end{align*}
    this would imply that $\abs{M_{\sigma}} = \abs{M_{\sigma'}}$ and therefore $\RtLmin(\sigma) = \RtLmin(\sigma')$.

    We will show that $\rho_M$ is a bijection in two steps. First, we will classify the right-to-left minima of $\sigma$ in terms of $\theta$ as being one of two types. Second, we will show that each type of right-to-left minima of $\sigma$ are in bijection with those of the same type for $\sigma'$.
    
    % Since $\Qsq$ is an involution, replacing $\theta$ by $\Qsq(\theta)$ reverses the roles of $\sigma$ and $\sigma'$.
    % Hence it will be sufficient to show that the restriction $\rho_M$ is well-defined, as running the same argument with $\theta$ replaced by $\Qsq(\theta)$ would show that its inverse is also well-defined and therefore that $\rho_M$ is a bijection.
    % This will happen in two steps: first, we will classify the right-to-left minima of $\sigma$ in terms of $\theta$, and second, we will show that this description applies when $\theta$ is replaced by $\Qsq(\theta)$ too.

    \begin{definition}
        We define $i \in [n]$ to be a \emph{split fixed point} of a permutation $\theta \in S_n$ if $i$ is both a right-to-left minimum and a left-to-right maximum of $\theta$.
        Equivalently, a split fixed point of $\theta$ is an $i \in [n]$ such that the OLN of $\theta$ factors as a concatenation $\alpha \cdot i \cdot \beta$ where $\alpha$ is a word in the letters $\{1, \dots, i-1\}$ and $\beta$ is a word in the letters $\{i+1, \dots, n\}$.

        Let $S(\theta)$ be the set of \emph{split fixed point pairs}:
        \begin{align*}
            S(\theta) &:= \{(i,i) : i \text{ is a split fixed point of } \theta\}.
        \end{align*}
    \end{definition}

    \begin{lemma}\label{lem.sfp_equiv_defs}
        Let $i \in [n]$, and let $\theta \in S_n$. Any two of the following imply the third:
        \begin{itemize}
            \item[(a)] $i$ is a right-to-left minimum of $\theta$,
            \item[(b)] $i$ is a left-to-right maximum of $\theta$,
            \item[(c)] $i$ is a fixed point of $\theta$.
        \end{itemize}
    \end{lemma}
    \begin{proof}
        If we assume (a) and (b), then the numbers to the left of $i$ in the OLN are smaller than $i$ and the numbers to the right of $i$ are larger than $i$.
        This implies that if a number is smaller than $i$, it appears to the left of $i$, and if a number is larger than $i$, it appears to the right of $i$.
        There are exactly $i-1$ numbers smaller than $i$, so $i$ is the $i$th number in the OLN from the left and therefore $\theta(i) = i$, making it a fixed point of $\theta$.

        If we assume (a) and (c), then the numbers to the right of $i$ are larger than $i$.
        We know that $i$ is in the $i$th place from the left, so there are exactly $n-i$ places to the right of $i$ which is equal to the size of $\{i+1, \dots, n\}$.
        Therefore all numbers larger than $i$ are to the right of $i$ with no slots left over.
        Hence all the numbers smaller than $i$ must be to the left of $i$, giving us that $i$ is a left-to-right maximum of $\theta$.

        If we assume (b) and (c), then the numbers to the left of $i$ are smaller than $i$.
        We know that $i$ is in the $i$th place from the left, so there are exactly $i-1$ places to the left of $i$ which is equal to the size of $\{1, \dots, i-1\}$.
        Therefore all numbers smaller than $i$ are to the left of $i$ with no slots left over.
        Hence all the numbers larger than $i$ must be to the right of $i$, giving us that $i$ is a right-to-left minimum of $\theta$.
    \end{proof}
    
    In our diagrams, a right-to-left minimum is characterized by having no other points to its southeast, and a left-to-right maximum is characterized by having no other points to its northwest.
    Therefore, a split fixed point is characterized by having no other points to its northwest or southeast.
    \begin{figure}[H]
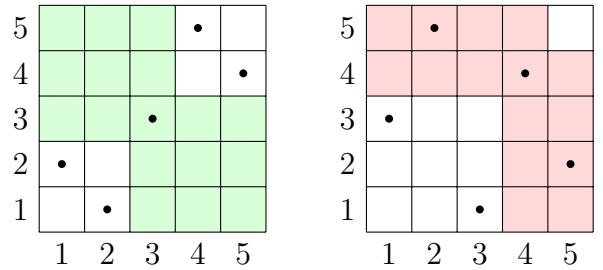

        \centering
        \begin{subfigure}[t]{0.3\textwidth}
            \centering
            \permgrid[scale=0.6, highlights={3/3/NWw/green!15, 3/3/SEw/green!15}]{5}{2,1,3,5,4}
            \caption*{Split fixed point at $3$}
        \end{subfigure}
        \begin{subfigure}[t]{0.3\textwidth}
            \centering
            \permgrid[scale=0.6, highlights={4/4/NWw/red!15, 4/4/SEw/red!15}]{5}{3,5,1,4,2}
            \caption*{No split fixed point at $4$}
        \end{subfigure}
    \end{figure}

    \begin{lemma}
        \label{lem.sfp_over_Finv_geometric}
        Let $i \in [n]$ and $\theta \in S_n$. Then, $i$ is a split fixed point of $\theta$ if and only if $i$ is a split fixed point of $\F^{-1}(\theta)$.
    \end{lemma}
    \begin{proof}
        If $i$ is a split fixed point of $\theta$, then by definition it is a left-to-right maximum of $\theta$ and a right-to-left minimum of $\theta$.
        If $i \neq n$, then since $i$ is a right-to-left minimum, $\theta(i+1) > i$ and is itself a left-to-right maximum of $\theta$:
        \begin{figure}[H]
            \centering
            \permgridpoints[scale=0.6, stagger columns, highlights={4/4/NWw/green!15, 5/6/NWw/green!15}]{7}{4/4/$i$/$i$, 5/6/$(i+1)$/$\theta(i+1)$}
            \caption*{Both $i$ and $\theta(i+1)$ are left-to-right maxima.}
        \end{figure}
        This means that $\F^{-1}(\theta)$ has $i$ in a cycle of its own and therefore that $i$ remains a fixed point.
        If $i = n$, then there are no further elements to be in a cycle with it, so $\F^{-1}(\theta)$ also keeps $i$ as a fixed point.
        The cycles to the left of $i$ consist of numbers less than $i$ while the cycles to the right of $i$ consist of numbers greater than $i$.
        Therefore we can decompose $\F^{-1}(\theta)$ into a permutation on $\{1, \dots, i-1\}$, a permutation on $\{i\}$, and a permutation on $\{i+1, \dots, n\}$.
        Writing this in OLN, we have $\F^{-1}(\theta) = \alpha \cdot i \cdot \beta$ where $\alpha$ and $\beta$ are words of $\{1, \dots, i-1\}$ and $\{i+1, \dots, n\}$ respectively.
        Thus $i$ is a split fixed point of $\F^{-1}(\theta)$.

        Conversely, if $i$ is a split fixed point of $\F^{-1}(\theta)$, then we can decompose $\F^{-1}(\theta)$ into a permutation on $\{1, \dots, i-1\}$, a permutation on $\{i\}$, and a permutation on $\{i+1, \dots, n\}$. In the CCD with cycles sorted by their maxima, that means that all cycles to the left of $i$ act on $\{1, \dots, i-1\}$ and all cycles to the right of $i$ act on $\{i+1, \dots, n\}$. Therefore, removing all parentheses by applying $\F$ leaves us with $\theta = \alpha \cdot i \cdot \beta$ where $\alpha$ and $\beta$ are words on $\{1, \dots, i-1\}$ and $\{i+1, \dots, n\}$ respectively, so we can conclude that $i$ is a split fixed point of $\theta$.
    \end{proof}

    \begin{corollary}
        Let $i \in [n]$ and $\theta \in S_n$. Then, $i$ is a split fixed point of $\theta$ if and only if $i$ is a split fixed point of $\F(\theta)$.
    \end{corollary}
    \begin{proof}
        Apply Lemma \ref{lem.sfp_over_Finv_geometric} to $\F(\theta)$ instead of $\theta$.
    \end{proof}

    \begin{definition}
        Let $\theta \in S_n$. We define the set of \emph{descent pairs} $E(\theta)$ to be
        \begin{align*}
            E(\theta) &:= \{(\theta(j), \theta(j+1)) : j \in [n-1], \theta(j) > \theta(j+1)\}.
        \end{align*}
    \end{definition}

    \begin{lemma}\label{lem.rtlmin-1_geometric}
        Let $\theta \in S_n$, and let $\sigma = \F^{-1}(\theta)$. Then:
        \begin{itemize}
            \item[(a)] If $(a,b) \in E(\theta)$, then $\sigma^{-1}(b) = a$.
            \item[(b)] $(\sigma^{-1}(b),b) \in E(\theta)$ if and only if $b$ is an excedance of $\sigma^{-1}$.
        \end{itemize}
    \end{lemma}
    \begin{proof}
        If $(a,b) \in E(\theta)$, then $a = \theta(j)$ and $b = \theta(j+1)$ for some $j \in [n-1]$ with $\theta(j) > \theta(j+1)$. Since $\theta(j+1) = b$ is smaller than $\theta(j)$, it cannot be a left-to-right maximum of $\theta$ and therefore is in the same cycle as $\theta(j)$ in $\sigma$. Thus $\sigma(\theta(j)) = \theta(j+1)$ which we can rewrite as $\sigma(a) = b$. Hence $\sigma^{-1}(b) = a$.
    
        For any $(\sigma^{-1}(b),b) \in E(\theta)$, by definition of $E(\theta)$, $\sigma^{-1}(b) > b$ and therefore $b$ is an excedance of $\sigma^{-1}$.

        Conversely, since every cycle begins with its maximum in CCD, the excedances of $\sigma^{-1}$ come in two forms: either they occur in the main body of a cycle or they occur going from the last element of the cycle to the first.
        In the first case, $\sigma^{-1}$ has a cycle that looks like $(\ldots, k, l, \ldots)$ where $k < l$ such that $k$ is an excedance.
        Since $\sigma^{-1}(k) = l$, we also have that $\sigma(l) = k$.
        Applying $\F$ gives an OLN:
        \begin{align*}
            \F(\sigma) = \theta &= \alpha \cdot l \cdot k \cdot \beta
        \end{align*}
        where $k < l$. Hence $(l,k) \in E(\theta)$.
        In the second case, $\sigma^{-1}$ has a cycle that looks like $(m, \ldots, r)$ where $r < m$ such that $r$ is an excedance.
        Since $\sigma^{-1}(r) = m$, we also have $\sigma(m) = r$.
        Also, the order of the cycle has reversed but $m$ is still the maximum, so it now looks like $(m, r, \ldots)$.
        Applying $\F$ gives an OLN:
        \begin{align*}
            \F(\sigma) = \theta &= \alpha \cdot m \cdot r \cdot \beta
        \end{align*}
        where $r < m$. Hence $(m,r) \in E(\theta)$.
    \end{proof}

    \begin{definition}
        Let $\theta \in S_n$. We define the set of \emph{right-to-left minimum descent pairs} $E^{\min}(\theta)$ to be
        \begin{align*}
            E^{\min}(\theta) := \{(a,b) \in E(\theta) : \nexists(a',b') \in E(\theta), b' < b < a < a'\} .
        \end{align*}
    \end{definition}

    \begin{lemma}\label{lem.rtlmin-2_geometric}
        Let $\theta \in S_n$, and let $\sigma = \F^{-1}(\theta)$. For any $(a,b) \in E(\theta)$, $b$ is a right-to-left minimum of $\sigma$ if and only if $(a,b) \in E^{\min}(\theta)$.
        %there does not exist any $(a',b') \in E(\theta)$ such that $b' < b < a < a'$.
    \end{lemma}
    \begin{proof}
        By Lemma \ref{lem.rtlmin-1_geometric} (a), we can rewrite our chain of inequalities as $b' < b < \sigma^{-1}(b) < \sigma^{-1}(b')$.
        Then since $\sigma^{-1}(b)$ is the position of $b$ in $\sigma$ and similarly $\sigma^{-1}(b')$ is the position of $b'$ in $\sigma$, the existence of any such $(a',b')$ pair would imply a number smaller than $b$ which is to the right of $b$ in $\sigma$ and therefore show that $b$ was not a right-to-left minimum.
        Hence, if $b$ is a right-to-left minimum of $\sigma$, then no such $(a',b') \in E(\theta)$ exists.

        Conversely, if no such $(a',b') \in E(\theta)$ exists, suppose that $b$ was not a right-to-left minimum of $\sigma$.
        This would imply the existence of some $c$ with position $\sigma^{-1}(c)$ such that $c < b < \sigma^{-1}(b) < \sigma^{-1}(c)$.
        However, this chain of inequalities includes $c < \sigma^{-1}(c)$, so $c$ is an excedance of $\sigma^{-1}$.
        By Lemma \ref{lem.rtlmin-1_geometric} (b), this puts $(\sigma^{-1}(c), c) \in E(\theta)$ which is a contradiction.
        Therefore we can conclude that if no such $(a',b') \in E(\theta)$ exists, then $b$ is a right-to-left minimum of $\sigma$.
    \end{proof}
    \begin{lemma}\label{lem.rtlmin_classification_geometric}
        Let $\theta \in S_n$, and let $\sigma = \F^{-1}(\theta)$. A given number $i \in [n]$ is a right-to-left minimum of $\sigma$ if and only if
        \begin{itemize}
            \item either $i$ is a split fixed point of $\theta$,
            \item or $(\sigma^{-1}(i), i) \in E^{\min}(\theta)$.
        \end{itemize}
    \end{lemma}
    \begin{proof}
        If $i$ is a right-to-left minimum of $\sigma$, then by Lemma \ref{lem.exc_cell_no_rtlmin_converse} we know that $\sigma^{-1}(i) \geq i$.
        
        If $\sigma^{-1}(i) = i$, then $i$ is a fixed point of $\sigma$.
        This gives us two of the conditions in Lemma \ref{lem.sfp_equiv_defs} which implies the third, that $i$ is a left-to-right maximum of $\sigma$ also holds.
        Therefore $i$ is a split fixed point of $\sigma$.
        By Lemma \ref{lem.sfp_over_Finv_geometric}, this tells us that $i$ is a split fixed point of $\theta$ also.

        If $\sigma^{-1}(i) > i$, then we have $i$ as an excedance of $\sigma^{-1}$.
        By Lemma \ref{lem.rtlmin-1_geometric}, this means that $(\sigma^{-1}(i), i) \in E(\theta)$.
        Then by Lemma \ref{lem.rtlmin-2_geometric}, $(\sigma^{-1}(i),i) \in E^{\min}(\theta)$.

        Conversely, if $i$ is a split fixed point of $\theta$, then $i$ is also a split fixed point of $\sigma$ by Lemma \ref{lem.sfp_over_Finv_geometric}.
        Then by the definition of a split fixed point, $i$ is a right-to-left minimum of $\sigma$.

        If $(\sigma^{-1}(i), i) \in E^{\min}(\theta)$, then by Lemma \ref{lem.rtlmin-2_geometric}, $i$ is a right-to-left minimum of $\sigma$.
    \end{proof}
    This gives us our classification of the right-to-left minima of $\sigma$ in terms of $\theta$.
    That is, $\psi_\sigma(M_\sigma) = E^{\min}(\theta) \sqcup S(\theta)$ since the $(i,\sigma(i))$ pairs with $i$ the position of a right-to-left minimum $\sigma(i)$ of $\sigma$ are classified by Lemma \ref{lem.rtlmin_classification_geometric} as either a descent pair or a split fixed point pair and vice versa. Similarly, $\psi_{\sigma'}(M_{\sigma'}) = E^{\min}(\Qsq(\theta)) \sqcup S(\Qsq(\theta))$.

    Let $\tau: [n] \times [n] \to [n] \times [n]$ be the map which pointwise implements $\Qsq$:
    \begin{align*}
        \tau(a,b) &:= (n+1-a, n+1-b).
    \end{align*}

    Applying $\rho$ to every point on the diagram of $\theta$ is the same as producing the diagram of $\D(\theta)$ or reflecting across the line $x+y=n+1$.
    Applying $\tau$ to every point on the diagram of $\theta$ is the same as producing the diagram of $\Qsq(\theta)$ or rotating by 180 degrees.
    \begin{lemma}\label{lem.transport_sfp_geometric}
        For any $\theta \in S_n$, all of the following are true:
        \begin{itemize}
            \item[(1)] $\tau(S(\theta)) = S(\Qsq(\theta))$,
            \item[(2)] $\rho|_{S(\theta)} : S(\theta) \to S(\Qsq(\theta))$ is a bijection.
        \end{itemize}
    \end{lemma}
    \begin{proof}
        Recall that the split fixed point pairs are the points on a permutation's diagram with no other points to the northwest or southeast.

        If we apply $\tau$ to both those regions for a split fixed point $i$, $NW_{i}$ and $SE_{i}$ respectively, then we get
        \begin{align*}
            \tau(NW_{i}) &= \{(n+1-(i-a),n+1-(i+b)) \in [n] \times [n] : a,b \in \mathbb{N}\} \\
            &= \{((n+1-i)+a,(n+1-i)-b) \in [n] \times [n] : a,b \in \mathbb{N}\} \\
            \tau(SE_{i}) &= \{(n+1-(i+a),n+1-(i-b)) \in [n] \times [n] : a,b \in \mathbb{N}\} \\
            &= \{((n+1-i)-a,(n+1-i)+b) \in [n] \times [n] : a,b \in \mathbb{N}\} .
        \end{align*}
        Thus $\tau(NW_{i})$ is the region southeast of $\tau(i,i) = (n+1-i,n+1-i)$ and $\tau(SE_{i})$ is the region northwest of $\tau(i,i) = (n+1-i,n+1-i)$.
        Applying $\tau$ produces the diagram of $\Qsq(\theta)$, so if there were no other points in those regions for $\theta$, then there will also be no other points in those regions for $\Qsq(\theta)$.
        Thus $(n+1-i,n+1-i)$ is a split fixed point of $\Qsq(\theta)$.
        That this is a bijection is immediate because $\tau$ is an involution.

        Since $\rho(i,i) = (n+1-i, n+1-i) = \tau(i,i)$, part (1) and part (2) are equivalent, so $\rho|_{S(\theta)} : S(\theta) \to S(\Qsq(\theta))$ is a bijection.
    \end{proof}

    For example, if $\theta = 451326798$, then we have split fixed points at $(6,6)$ and at $(7,7)$.
    \begin{figure}[H]
        \centering
        \begin{subfigure}[t]{0.45\textwidth}
            \centering
            \permgrid[scale=0.6, highlights={6/6/NWw/green!15, 6/6/SEw/blue!15}]{9}{4,5,1,3,2,6,7,9,8}
            \caption*{Split fixed point at $(6,6)$}
        \end{subfigure}
        \hfill
        \begin{subfigure}[t]{0.45\textwidth}
            \centering
            \permgrid[scale=0.6, highlights={7/7/NWw/green!15, 7/7/SEw/blue!15}]{9}{4,5,1,3,2,6,7,9,8}
            \caption*{Split fixed point at $(7,7)$}
        \end{subfigure}
    \end{figure}

    Applying $\tau$ gives the diagram of $\Qsq(\theta) = 213487956$ and transports the split fixed points and the regions as pictured:
    
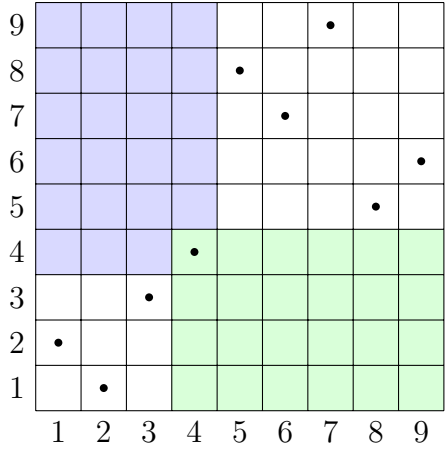
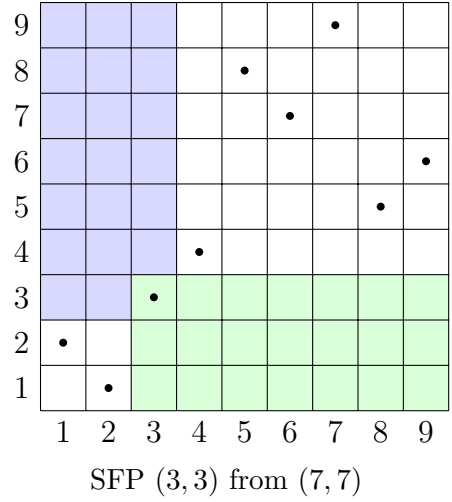
\begin{figure}[H]
        \centering
        \begin{subfigure}[t]{0.45\textwidth}
            \centering
            \permgrid[scale=0.6, highlights={4/4/NWw/blue!15, 4/4/SEw/green!15}]{9}{2,1,3,4,8,7,9,5,6}
            \caption*{SFP $(4,4)$ from $(6,6)$}
        \end{subfigure}
        \hfill
        \begin{subfigure}[t]{0.45\textwidth}
            \centering
            \permgrid[scale=0.6, highlights={3/3/NWw/blue!15, 3/3/SEw/green!15}]{9}{2,1,3,4,8,7,9,5,6}
            \caption*{SFP $(3,3)$ from $(7,7)$}
        \end{subfigure}
    \end{figure}

    \begin{lemma}\label{lem.transport_rtlmin_geometric}
        For any $\theta \in S_n$, all of the following are true:
        \begin{itemize}
            \item[(1)] $\rho|_{E(\theta)} : E(\theta) \to E(\Qsq(\theta))$ is a bijection,
            \item[(2)] $(x,y) \in E^{\min}(\theta)$ if and only if $\rho(x,y) = (q,r) \in E^{\min}(\Qsq(\theta))$,
            \item[(3)] $\rho|_{E^{\min}(\theta)} : E^{\min}(\theta) \to E^{\min}(\Qsq(\theta))$ is a bijection.
        \end{itemize}
    \end{lemma}
    \begin{proof}
        For $j \in [n-1]$, if $(\theta(j),\theta(j+1)) \in E(\theta)$, then $\rho(\theta(j),\theta(j+1)) = (n+1-\theta(j+1),n+1-\theta(j))$ and we claim that this point is in $E(\Qsq(\theta))$.
        We calculate that $\Qsq(\theta)(n+1-j) = n+1-\theta(j)$ and $\Qsq(\theta)(n+1-(j+1)) = n+1-\theta(j+1)$.
        Therefore
        \begin{align*}
            \rho(\theta(j),\theta(j+1)) &= (\Qsq(\theta)(n+1-(j+1)), \Qsq(\theta)(n+1-j)) \\
            &= (\Qsq(\theta)(n-j), \Qsq(\theta)(n-j+1)).
        \end{align*}
        Then $\theta(j) > \theta(j+1)$, so $n+1-\theta(j) < n+1-\theta(j+1)$.
        Hence $\Qsq(\theta)(n-j) > \Qsq(\theta)(n-j+1)$, so $\rho(\theta(j), \theta(j+1)) \in E(\Qsq(\theta))$.

        Replacing $\theta$ with $\Qsq(\theta)$ proves the reverse, so we can conclude part (1): that $\rho|_{E(\theta)}$ is a bijection between $E(\theta)$ and $E(\Qsq(\theta))$.

        For part (2), we consider the diagram of points in $E(\theta)$ (instead of the diagram of some permutation).
        The points $(x,y) \in E(\theta)$ for which there does not exist any $(x',y') \in E(\theta)$ such that $y' < y < x < x'$ are the points in the aforementioned diagram which have no other points to their southeast.

        By part (1), $\rho(E(\theta)) = E(\Qsq(\theta))$. 
        For a fixed $(x,y) \in E^{\min}(\theta)$, we define
        the southeast region $SE_{x,y}$:
        \begin{align*}
            SE_{x,y} &:= \{(x + a, y - b) \in [n] \times [n] : a,b \in \mathbb{N}\}.
        \end{align*}
        We know by hypothesis that $SE_{x,y} \cap E(\theta) = \{(x,y)\}$.
        Then if we apply $\rho$ to this set, we get
        \begin{align*}
            \rho(SE_{x,y}) &= \{(n+1-(y-b),n+1-(x+a)) \in [n]\times[n] : a,b \in \mathbb{N}\} \\
            &= \{((n+1-y)+b, (n+1-x)-a) \in [n] \times [n] : a,b\in\mathbb{N}\}.
        \end{align*}
        which is exactly the region southeast of $\rho(x,y) = (n+1-y,n+1-x) =: (q,r)$ in $\rho(E(\theta)) = E(\Qsq(\theta))$. Therefore $\rho(SE_{x,y}) \cap E(\Qsq(\theta)) = \{(q,r)\}$, so there does not exist any $(q',r') \in E(\Qsq(\theta))$ such that $r' < r < q < q'$.

        Again, replacing $\theta$ with $\Qsq(\theta)$ proves the reverse, so this proves part (2).

        Taking parts (1) and (2) together, further restricting $\rho$ to $E^{\min}(\theta)$ will then give us a bijection between $E^{\min}(\theta)$ and $E^{\min}(\Qsq(\theta))$ which is part (3).
    \end{proof}

%\permgridpoints[scale=0.6, labelscale=0.6, highlights={4/4/NWw/green!15, 5/6/NWw/green!15}]{7}{4/4/$i$/$i$, 5/6/$(i+1)$/$\theta(i+1)$}

    For example, if $\theta = 451326798$, then the set of descent pairs is $E(\theta) = \{(5,1), (3,2), (9,8)\}$. From those, $1 < 2 < 3 < 5$, so $E^{\min}(\theta) = \{(5,1),(9,8)\}$.
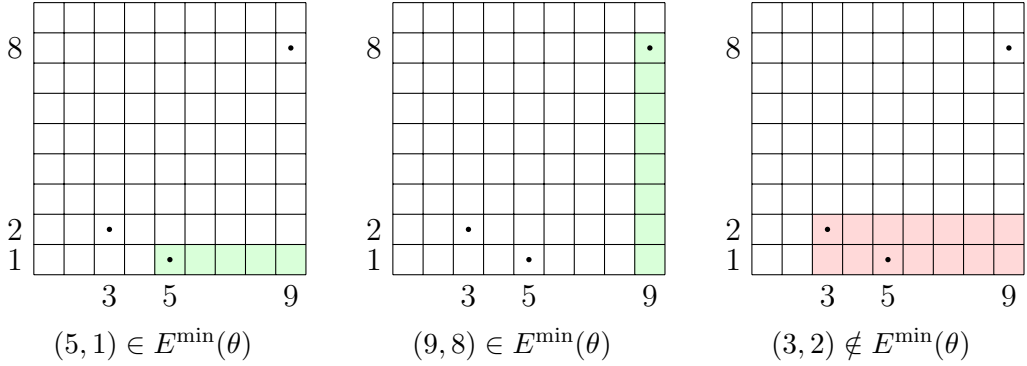
    \begin{figure}[H]
        \centering
        \begin{subfigure}[t]{0.32\textwidth}
            \centering
            \permgridpoints[scale=0.4, labelscale=1, highlights={5/1/SEw/green!15}]{9}{5/1/$5$/$1$, 3/2/$3$/$2$, 9/8/$9$/$8$}
            \caption*{$(5,1) \in E^{\min}(\theta)$}
        \end{subfigure}
        \hfill
        \begin{subfigure}[t]{0.32\textwidth}
            \centering
            \permgridpoints[scale=0.4, labelscale=1, highlights={9/8/SEw/green!15}]{9}{5/1/$5$/$1$, 3/2/$3$/$2$, 9/8/$9$/$8$}
            \caption*{$(9,8) \in E^{\min}(\theta)$}
        \end{subfigure}
        \hfill
        \begin{subfigure}[t]{0.32\textwidth}
            \centering
            \permgridpoints[scale=0.4, labelscale=1, highlights={3/2/SEw/red!15}]{9}{5/1/$5$/$1$, 3/2/$3$/$2$, 9/8/$9$/$8$}
            \caption*{$(3,2) \notin E^{\min}(\theta)$}
        \end{subfigure}
    \end{figure}
    Then applying $\rho$ gives us the following diagram, preserving the right-to-left minimum descent pairs.
    %FORMATTING? Maybe this is too many diagrams? Maybe I just need to get the captions looking nicer?
    
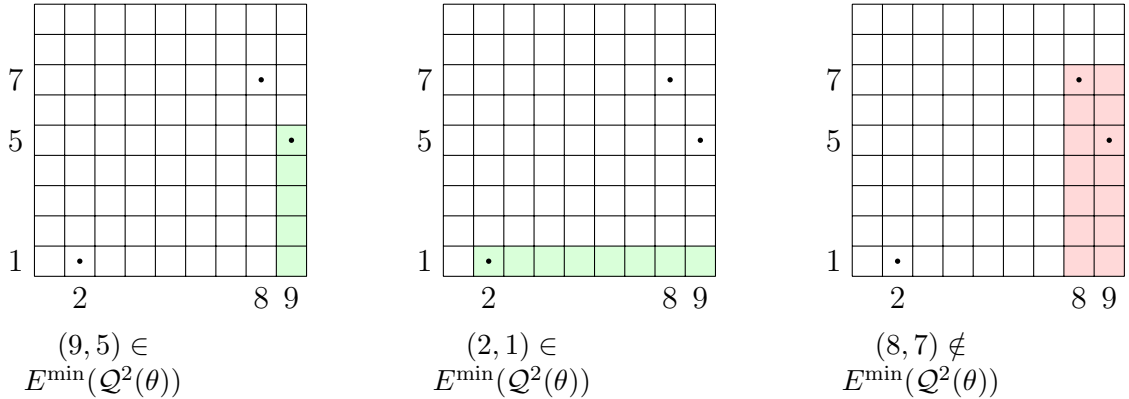
\begin{figure}[H]
        \centering
        \begin{subfigure}[t]{0.2\textwidth}
            \centering
            \permgridpoints[scale=0.4, labelscale=1, highlights={9/5/SEw/green!15}]{9}{9/5/$9$/$5$, 8/7/$8$/$7$, 2/1/$2$/$1$}
            \caption*{$(9,5) \in E^{\min}(\Qsq(\theta))$}
        \end{subfigure}
        \qquad\qquad\qquad
        \begin{subfigure}[t]{0.2\textwidth}
            \centering
            \permgridpoints[scale=0.4, labelscale=1, highlights={2/1/SEw/green!15}]{9}{9/5/$9$/$5$, 8/7/$8$/$7$, 2/1/$2$/$1$}
            \caption*{$(2,1) \in E^{\min}(\Qsq(\theta))$}
        \end{subfigure}
        \qquad\qquad\qquad
        \begin{subfigure}[t]{0.2\textwidth}
            \centering
            \permgridpoints[scale=0.4, labelscale=1, highlights={8/7/SEw/red!15}]{9}{9/5/$9$/$5$, 8/7/$8$/$7$, 2/1/$2$/$1$}
            \caption*{$(8,7) \notin E^{\min}(\Qsq(\theta))$}
        \end{subfigure}
    \end{figure}
    
    %Put things together to get our final conclusion
    We now apply our accumulated results to prove the proposition. By Lemma \ref{lem.transport_sfp_geometric} part (2), we have a bijection $\rho|_{S(\theta)} : S(\theta) \to S(\Qsq(\theta))$, and by Lemma \ref{lem.transport_rtlmin_geometric} part (3), we have a bijection $\rho|_{E^{\min}(\theta)}: E^{\min}(\theta) \to E^{\min}(\Qsq(\theta))$.
    Then by Lemma \ref{lem.rtlmin_classification_geometric}, $\rho|_{S(\theta)}$ and $\rho|_{E^{\min}(\theta)}$ are bijections between pairs which correspond to right-to-left minima of $\sigma$ to pairs which correspond to right-to-left minima of $\sigma'$.
    Hence $\rho_M : \psi_\sigma(M_\sigma) \to \psi_{\sigma'}(M_{\sigma'})$ can be represented:
    \begin{align*}
        \rho_M(x,y) &= \begin{cases} \rho|_{S(\theta)}(x,y) & (x,y) \in S(\theta) \\
        \rho|_{E(\theta)}(x,y) & (x,y) \in E^{\min}(\theta)
        \end{cases} \\
        &= (n+1-y,n+1-x) .
    \end{align*}
    and is a bijection.
    This proves that $|\psi_\sigma(M_\sigma)| = |\psi_{\sigma'}(M_{\sigma'})|$, so $\RtLmin(\sigma) = \RtLmin(\sigma')$.
\end{proof}

\subsection{Proof of Theorem \ref{thm.mainD}}

\begin{proof}[\hypertarget{pf.des_ends}{\textbf{Proof of Theorem \ref{thm.mainD}}}]
    To prove that $D_1 - D_{n-1}$ is $0$-mesic under the action of rotation-reversal, we use Lemma \ref{lem.telescope_strategy} with $f(\pi) = D_1(\pi) - D_1(\I(\pi))$.

    Therefore we want to prove that for any $\pi \in S_n$, the following holds:
    \begin{align*}
        &D_1(\pi) - D_{n-1}(\pi) \\
        &= D_1(\pi) - D_1(\I(\pi)) - D_1(\R\circ\F^{-1}\circ\Qsq\circ\F(\pi)) + D_1(\I\circ\R\circ\F^{-1}\circ\Qsq\circ\F(\pi)).
    \end{align*}
    Or equivalently, after cancellation and some rearrangement,
    $$D_1(\I(\pi)) + D_1(\R\circ\F^{-1}\circ\Qsq\circ\F(\pi)) = D_{n-1}(\pi) + D_1(\I\circ\R\circ\F^{-1}\circ\Qsq\circ\F(\pi)).$$

    Writing $U = \F^{-1} \circ \Qsq \circ \F$, this rewrites to
    $$D_1(\I(\pi)) + D_1(\R\circ U(\pi)) = D_{n-1}(\pi) + D_1(\I\circ\R\circ U(\pi)).$$

    We will show that both sides of this equation are equal to $1$ by proving the equivalent identity
    $$D_1(\I \circ U(\theta)) = D_{n-1}(\theta)$$
    for all $\theta \in S_n$.

    The equivalence will be in Lemma \ref{lem.D_equiv} and then the identity will be in Lemma \ref{lem.D_identity}.
\end{proof}

\begin{lemma}
    \label{lem.reversal_facts}
    For every $\theta \in S_n$ with $n \geq 2$,
    \begin{itemize}
        \item[(a)] $D_1(\R(\theta)) = 1 - D_{n-1}(\theta),$ and
        \item[(b)] $D_1(\I(\theta)) + D_1(\I \circ \R(\theta)) = 1.$
    \end{itemize}
\end{lemma}
\begin{proof}
    For (a), we can use the definition of $D_i$ to calculate that
    \begin{align*}
        D_1(\R(\theta)) &= [\R(\theta)(1) > \R(\theta)(2)] \\
        &= [\theta(n) > \theta(n-1)] \\
        &= 1 - [\theta(n) < \theta(n-1)] = 1 - D_{n-1}(\theta).
    \end{align*}

    For (b) and each $i \in [n]$, we start by noting that reversing $\theta$ moves $i$ from position $\I(\theta)(i)$ to position $n + 1 - \I(\theta)(i)$, so we have the corresponding identity $\I(\theta)(i) = n + 1 - \I \circ \R(\theta)(i)$.
    Let us write $\tau = \I \circ \R(\theta)$.
    We can again use the definition of $D_i$ and the above identity to get:
    \begin{align*}
        &D_1(\I(\theta)) + D_1(\I \circ \R(\theta)) \\
        &= [\I(\theta)(1) > \I(\theta)(2)] + [\tau(1) > \tau(2)] \\
        &= [n+1-\tau(1) > n+1-\tau(2)] + [\tau(1) > \tau(2)] \\
        &= [\tau(1) < \tau(2)] + [\tau(1) > \tau(2)].
    \end{align*}
    Then since $\tau(1) \neq \tau(2)$, exactly one of these two will always be true, so their sum is $1$.
\end{proof}

\begin{lemma}
    \label{lem.D_equiv}
    Let $n \geq 2$. Recall that $U = \F^{-1} \circ \Qsq \circ \F$. The following are equivalent:
    \begin{itemize}
        \item[(a)] For every $\theta \in S_n$, $D_1(\I(\theta)) + D_1(\R\circ U(\theta)) = 1$;
        \item[(b)] For every $\theta \in S_n$, $D_{n-1}(\theta) + D_1(\I\circ\R\circ U(\theta)) = 1$;
        \item[(c)] For every $\theta \in S_n$, $D_1(\I \circ U(\theta)) = D_{n-1}(\theta)$.
    \end{itemize}
\end{lemma}
\begin{proof}
    First for showing that (a) and (c) are equivalent, we note that $U \circ U$ is the identity. Then if we substitute $\theta$ with $U(\pi)$, we can rewrite (a) with
    \begin{align*}
        D_1(\I\circ U(\pi)) + D_1(\R(\pi)) = 1
    \end{align*}
    since $\R\circ U \circ U(\pi) = \R(\pi)$. Then we apply Lemma \ref{lem.reversal_facts} (a) to obtain
    \begin{align*}
        D_1(\I\circ U(\pi)) + 1 - D_{n-1}(\pi) &= 1 \\
        D_1(\I\circ U(\pi)) &= D_{n-1}(\pi)
    \end{align*}
    which is (c). All of these steps are reversible, so we have equivalence.

    Next, we show that (b) and (c) are equivalent. We'll substitute $\theta$ with $U(\pi)$ to turn (b) into
    \begin{align*}
        D_{n-1}(U(\pi)) + D_1(\I\circ\R(\pi)) = 1 .
    \end{align*}
    Then we apply Lemma \ref{lem.reversal_facts} (b) to obtain
    \begin{align*}
        D_{n-1}(U(\pi)) + (1 - D_1(\I(\pi))) &= 1 \\
        D_{n-1}(U(\pi)) &= D_1(\I(\pi)) .
    \end{align*}
    Finally, since $\theta = U(\pi)$, we also have $U(\theta) = U\circ U(\pi) = \pi$. Using this substitution we get
    \begin{align*}
        D_{n-1}(\theta) &= D_1(\I\circ U (\theta))
    \end{align*}
    which is (c). Again, all of these steps are reversible, so we have equivalence.

    Then (a) and (b) are both equivalent to (c), so they are equivalent to each other also.
\end{proof}

\begin{lemma}
    \label{lem.D_identity}
    Let $n \geq 2$. For every $\theta \in S_n$,
    $$D_1(\I \circ U(\theta)) = D_{n-1}(\theta).$$
\end{lemma}
\begin{proof}
    We will break things into three cases:
    \begin{itemize}
        \item[(i)] Both $\theta(n) \neq n$ and $n-1$ is not the last element in its cycle in $\theta$.
        \item[(ii)] $\theta(n) = n$.
        \item[(iii)] Both $\theta(n) \neq n$ and $n-1$ is the last element in its cycle in $\theta$.
    \end{itemize}

    First, if $\I \circ U(\theta)(1) = n+1-\theta(n)$ and $\I \circ U (\theta)(2) = n+1-\theta(n-1)$, then $D_1(\I \circ U(\theta)) = [\I \circ U(\theta)(1) > \I \circ U(\theta)(2)] = [\theta(n-1) > \theta(n)] = D_{n-1}(\theta)$.
    We claim that $\I \circ U(\theta)(1) = n+1-\theta(n)$ and $\I \circ U (\theta)(2) = n+1-\theta(n-1)$ hold in case (i).
    
    So in case (i) with both of those two conditions, we note that the CCD of $\theta$ has some element to the right of each of $n$ and $n-1$ in their respective cycles (possibly the same cycle).
    This means that $\theta(n-1) \notin \{n-1,n\}$ and $\theta(n) \neq n$.
    Applying $\Qsq \circ \F$ to that $\theta$ gives us the elements $n+1-\theta(n)$ and $n+1-\theta(n-1)$ immediately to the left of $1$ and $2$ respectively.
    Then we have $n+1-\theta(n-1) > 2$, and $n+1-\theta(n) > 1$, so neither $1$ nor $2$ can be a left-to-right maximum of $\Qsq \circ \F(\theta)$.
    Therefore applying $\F^{-1}$ gives $U(\theta)(n+1-\theta(n)) = 1$ and $U(\theta)(n+1-\theta(n-1)) = 2$.
    Thus $\I \circ U(\theta)(1) = n+1-\theta(n)$ and $\I \circ U (\theta)(2) = n+1-\theta(n-1)$.

    In case (ii) with $\theta(n) = n$, then $D_{n-1}(\theta) = 0$ since $\theta(n-1) \neq n$ and is therefore less than $n$.
    We then calculate that
    \begin{align*}
        \theta(n) &= n \\
        &\iff \F(\theta)(n) = n \\
        &\iff \Qsq\circ\F(\theta)(1) = 1 \\
        &\iff U(\theta)(1) = 1 \\
        &\iff \I\circ U(\theta)(1) = 1.
    \end{align*}
    Then $\I \circ U(\theta)(2) \neq 1$ and therefore must be larger than $1$, so $D_1(\I \circ U(\theta)) = 0 = D_{n-1}(\theta)$.

    In case (iii) where $n-1$ is the last element in its cycle in $\theta$ and $\theta(n) \neq n$, then either $\theta(n-1) = n-1$, or $\theta(n-1) = n$ since $\theta(n-1)$ is the first and largest element of the cycle containing $n-1$.
    If $\theta(n-1) = n$, then $\theta(n-1) > \theta(n)$ immediately.
    If $\theta(n-1) = n-1$, then $\theta(n-1) > \theta(n)$ by injectivity (since $\theta(n) \neq n$).
    In both of these cases, $D_{n-1}(\theta) = 1$.
    
    In the case of $\theta(n-1) = n-1$, the CCD of $\theta$ is of the form $\dots(n-1)(n,\dots)$, so $\Qsq \circ \F(\theta) = \alpha \cdot 1 \cdot 2 \cdot \beta$ for some nonempty word $\alpha$ and a possibly empty word $\beta$.
    Since $\alpha$ is nonempty, it must contain some element greater than $2$, so $2$ is not a left-to-right maximum.
    Then $U(\theta) = \dots(\dots,1,2,\dots)\dots$, so $U(\theta)(1) = 2$.
    Hence $\I\circ U(\theta)(2) = 1$ and $\I\circ U(\theta)(1) > 1$, so $D_1(\I \circ U(\theta))=1=D_{n-1}(\theta)$.

    In the case of $\theta(n-1) = n$, the CCD of $\theta$ is of the form $\dots(n, n-1)$ or of the form $\dots(n, \dots, n-1)$ with some elements in between, so $\Qsq \circ \F(\theta) = 2 \cdot \alpha \cdot 1 \cdot \beta$ where $\alpha$ is empty or nonempty respectively.
    If $\alpha$ is empty, then whatever element begins $\beta$ is greater than $2$, so it is a left-to-right maximum of $\Qsq\circ\F(\theta)$. Therefore the first cycle of $U(\theta)$ is $(2,1)$. If $\beta$ is itself empty, then we have $U(\theta) = (2,1)$, so the first cycle is still $(2,1)$.
    Therefore $U(\theta)(2) = 1$ and $U(\theta)(1) = 2$, so $\I \circ U(\theta)(1) = 2 > 1 = \I\circ U(\theta)(2)$ and $D_1(\I \circ U(\theta)) = 1 = D_{n-1}(\theta)$.
    If $\alpha$ is nonempty, then the first element of $\alpha$ is greater than $2$, so it is a left-to-right maximum of $\Qsq\circ\F(\theta)$.
    Therefore $U(\theta)(2) = 2$ and $U(\theta)(1)\neq 1$ (since $1$ is not a left-to-right maximum), so $\I \circ U(\theta)(2) = 2$ and $\I \circ U(\theta)(1) \neq 1$.
    Hence $\I \circ U(\theta)(1) > \I \circ U(\theta)(2)$ by injectivity, so $D_1(\I \circ U(\theta)) = 1 = D_{n-1}(\theta)$.

\end{proof}

\subsection{Proof of Theorem \ref{thm.mainFixEnds}}

\begin{proof}[\hypertarget{pf.fix_ends}{\textbf{Proof of Theorem \ref{thm.mainFixEnds}}}]
    To prove that $\Fix_1 - \Fix_n$ is $0$-mesic under the actions of rotated inverse-complement, complement-rotated inverse, rotated inverse-reversal, and reversal-rotated inverse, we use Lemma \ref{lem.telescope_strategy} separately in two different cases.

    For rotated inverse-reversal and reversal-rotated inverse, we have $f(\pi) = -[\pi(1) = n] - \Fix_n(\pi)$ using Iverson brackets.

    For rotated inverse-complement and complement-rotated inverse, we have $f(\pi) = \Fix_1(\pi) + [\pi(1) = n]$.

    For rotated inverse-complement and complement-rotated inverse (written as $\tau$), we rewrite the claim that we want to prove:
    \begin{align*}
        \Fix_1(\pi) - \Fix_n(\pi) &= f(\pi) - f(\tau(\pi)) \\
        \Fix_1(\pi) - \Fix_n(\pi) &= \Fix_1(\pi) + [\pi(1) = n] - (\Fix_1(\tau(\pi)) + [\tau(\pi)(1) = n]) \\
        -\Fix_n(\pi) &= [\pi(1) = n] - \Fix_1(\tau(\pi)) - [\tau(\pi)(1) = n] \\
        \Fix_n(\pi) + [\pi(1) = n] &= \Fix_1(\tau(\pi)) + [\tau(\pi)(1) = n] .
    \end{align*}
    We prove the claim for rotated inverse-complement in Lemma \ref{lem.DC} and we prove the claim for complement-rotated inverse in Lemma \ref{lem.CD}.

    For rotated inverse-reversal and reversal-rotated inverse (written as $\tau$), we rewrite the claim that we want to prove:
    \begin{align*}
        \Fix_1(\pi) - \Fix_n(\pi) &= f(\pi) - f(\tau(\pi)) \\
        \Fix_1(\pi) - \Fix_n(\pi) &= -[\pi(1) = n] - \Fix_n(\pi) - (-[\tau(\pi)(1) = n] - \Fix_n(\tau(\pi))) \\
        \Fix_1(\pi) &= -[\pi(1) = n] + [\tau(\pi)(1) = n] + \Fix_n(\tau(\pi)) \\
        \Fix_1(\pi) + [\pi(1) = n] &= [\tau(\pi)(1) = n] + \Fix_n(\tau(\pi)).
    \end{align*}
    We prove the claim for rotated inverse-reversal in Lemma \ref{lem.DR} and we prove the claim for reversal-rotated inverse in Lemma \ref{lem.RD}.
\end{proof}
\begin{lemma}
    \label{lem.DC}
    For $\tau = \C \circ \F^{-1} \circ \D \circ \F$ and every $\pi \in S_n$,
    \begin{align*}
        \Fix_n(\pi) &= [\tau(\pi)(1) = n] \\
        &\text{and} \\
        [\pi(1) = n] &= \Fix_1(\tau(\pi)).
    \end{align*}
\end{lemma}
\begin{proof}
    To show that $\Fix_n(\pi) = [\tau(\pi)(1) = n]$, we calculate the consequences of each map being applied:
    \begin{align*}
        \Fix_n(\pi) = 1 & \iff \pi(n) = n\\
        &\iff \F(\pi)(n) = n \\
        &\iff \I \circ \F(\pi)(n) = n \\
        &\iff \Qsq \circ \I \circ \F(\pi)(1) = 1 \\
        &\iff \F^{-1}\circ\D\circ\F(\pi)(1) = 1 \\
        &\iff \C\circ\F^{-1} \circ \D \circ \F(\pi)(1) = n \\
        &\iff [\tau(\pi)(1) = n] = 1.
    \end{align*}
    Similarly, we want to show that $[\pi(1) = n] = \Fix_1(\tau(\pi))$. To start things off, we note that if $\pi(1) = n$, then in the CCD, $n$ always begins the last cycle. Then for $\pi(1) = n$, we can conclude that the last cycle ends with $1$. Thus $\F(\pi)(n) = 1$. Hence we calculate
    \begin{align*}
        [\pi(1) = n] = 1 &\iff \pi(1) = n \\
        &\iff \F(\pi)(n) = 1 \\
        &\iff \I \circ \F(\pi)(1) = n \\
        &\iff \Qsq \circ \I \circ \F(\pi)(n) = 1.
    \end{align*}
    Now if we apply $\F^{-1}$, since $1$ is the last entry in the OLN of $\D \circ \F(\pi)$, it must be in the last cycle. That last cycle must always begin with $n$, so $\F^{-1}\circ\D\circ\F(\pi)(1) = n$. Thus
    \begin{align*}
        [\pi(1) = n] = 1 &\iff \D \circ \F(\pi)(n) = 1 \\
        &\iff \F^{-1}\circ\D\circ\F(\pi)(1) = n \\
        &\iff \C \circ \F^{-1}\circ\D\circ\F(\pi)(1) = 1 \\
        &\iff \tau(\pi)(1) = 1 \\
        &\iff \Fix_1(\tau(\pi)) = 1.
    \end{align*}
\end{proof}
\begin{lemma}
    \label{lem.CD}
    For $\tau = \D \circ \F^{-1} \circ \C \circ \F$ and every $\pi \in S_n$,
    \begin{align*}
        \Fix_n(\pi) &= [\tau(\pi)(1) = n] \\
        &\text{and} \\
        [\pi(1) = n] &= \Fix_1(\tau(\pi)).
    \end{align*}
\end{lemma}
\begin{proof}
    To show that $\Fix_n(\pi) = [\tau(\pi)(1) = n]$, we calculate that
    \begin{align*}
        \Fix_n(\pi) = 1 &\iff \pi(n) = n \\
        &\iff \F(\pi)(n) = n \\
        &\iff \C \circ \F(\pi)(n) = 1 \\
        &\iff \F^{-1} \circ \C \circ \F(\pi)(1) = n \\
        &\iff \I \circ \F^{-1} \circ \C \circ \F(\pi)(n) = 1 \\
        &\iff \Qsq \circ \I \circ \F^{-1} \circ \C \circ \F(\pi)(1) = n \\
        &\iff \tau(\pi)(1) = n \\
        &\iff [\tau(\pi)(1) = n] = 1.
    \end{align*}
    Then to show that $[\pi(1) = n] = \Fix_1(\tau(\pi))$, we have
    \begin{align*}
        [\pi(1) = n] = 1 &\iff \pi(1) = n \\
        &\iff \F(\pi)(n) = 1 \\
        &\iff \C \circ \F(\pi)(n) = n \\
        &\iff \F^{-1} \circ \C \circ \F(\pi)(n) = n \\
        &\iff \I \circ \F^{-1} \circ \C \circ \F(\pi)(n) = n \\
        &\iff \Qsq \circ \I \circ \F^{-1} \circ \C \circ \F(\pi)(1) = 1 \\
        &\iff \tau(\pi)(1) = 1 \\
        &\iff \Fix_1(\tau(\pi)) = 1.
    \end{align*}
\end{proof}
\begin{lemma}
    \label{lem.DR}
    For $\tau = \R \circ \F^{-1} \circ \D \circ \F$ and every $\pi \in S_n$,
    \begin{align*}
        \Fix_1(\pi) &= [\tau(\pi)(1) = n] \\
        &\text{and} \\
        [\pi(1) = n] &= \Fix_n(\tau(\pi)).
    \end{align*}
\end{lemma}
\begin{proof}
    To show that $\Fix_1(\pi) = [\tau(\pi)(1) = n]$, we calculate that
    \begin{align*}
        \Fix_1(\pi) = 1 &\iff \pi(1) = 1 \\
        &\iff \F(\pi)(1) = 1 \\
        &\iff \I \circ \F(\pi)(1) = 1 \\
        &\iff \Qsq \circ \I \circ \F(\pi)(n) = n \\
        &\iff \F^{-1} \circ \D \circ \F(\pi)(n) = n \\
        &\iff \R \circ \F^{-1} \circ \D \circ \F(\pi)(1) = n \\
        &\iff \tau(\pi)(1) = n \\
        &\iff [\tau(\pi)(1) = n] = 1.
    \end{align*}
    To show that $[\pi(1) = n] = \Fix_n(\tau(\pi))$, we calculate that
    \begin{align*}
        [\pi(1) = n] = 1 &\iff \pi(1) = n \\
        &\iff \F(\pi)(n) = 1 \\
        &\iff \I \circ \F(\pi)(1) = n \\
        &\iff \Qsq \circ \I \circ \F(\pi)(n) = 1 \\
        &\iff \F^{-1} \circ \D \circ \F(\pi)(1) = n \\
        &\iff \R \circ \F^{-1} \circ \D \circ \F(\pi)(n) = n \\
        &\iff \tau(\pi)(n) = n \\
        &\iff \Fix_n(\tau(\pi)) = 1.
    \end{align*}
\end{proof}
\begin{lemma}
    \label{lem.RD}
    For $\tau = \D \circ \F^{-1} \circ \R \circ \F$ and every $\pi \in S_n$,
    \begin{align*}
        \Fix_1(\pi) &= [\tau(\pi)(1) = n] \\
        &\text{and} \\
        [\pi(1) = n] &= \Fix_n(\tau(\pi)).
    \end{align*}
\end{lemma}
\begin{proof}
    To show that $\Fix_1(\pi) = [\tau(\pi)(1) = n]$, we calculate that
    \begin{align*}
        \Fix_1(\pi) = 1 &\iff \pi(1) = 1 \\
        &\iff \F(\pi)(1) = 1 \\
        &\iff \R \circ \F(\pi)(n) = 1 \\
        &\iff \F^{-1} \circ \R \circ \F(\pi)(1) = n \\
        &\iff \I\circ\F^{-1}\circ\R\circ\F(\pi)(n) = 1 \\
        &\iff \Qsq\circ\I\circ\F^{-1}\circ\R\circ\F(\pi)(1) = n \\
        &\iff \D\circ\F^{-1}\circ\R\circ\F(\pi)(1) = n \\
        &\iff \tau(\pi)(1) = n \\
        &\iff [\tau(\pi)(1) = n] = 1.
    \end{align*}
    To show that $[\pi(1) = n] = \Fix_n(\tau(\pi))$, we calculate that
    \begin{align*}
        [\pi(1) = n] = 1 &\iff \pi(1) = n \\
        &\iff \F(\pi)(n) = 1 \\
        &\iff \R \circ \F(\pi)(1) = 1 \\
        &\iff \F^{-1}\circ\R\circ\F(\pi)(1) = 1 \\
        &\iff \I\circ\F^{-1}\circ\R\circ\F(\pi)(1) = 1 \\
        &\iff \Qsq\circ\I\circ\F^{-1}\circ\R\circ\F(\pi)(n) = n \\
        &\iff \tau(\pi)(n) = n \\
        &\iff \Fix_n(\tau(\pi)) = 1.
    \end{align*}
\end{proof}

\section{\label{sect.proved_tables}Summarizing Results}

We here summarize the known telescoping functions associated with certain Foatic map and statistic pairs. Some results are due to the work done in this paper while others are easy corollaries of work done previously.

\begin{corollary}
    \label{cor.wexc_sr}
    In proving the $\frac{n+1}{2}$-mesicity of $\wexc$ under the action of reversal-rotated inverse and complement-rotated inverse, where $T$ denotes either action, Sheridan-Rossi \cite[Theorem 2.2.12]{Sheridan-Rossi2020} shows that
    $$\wexc(T(\pi)) = n + 1 - \wexc(\pi)$$
    which we can rewrite as
    \begin{align*}
        \wexc(T(\pi)) + \wexc(\pi) &= n + 1 \\
        \frac{\wexc(T(\pi)) + \wexc(\pi)}{2} &= \frac{n+1}{2} \\
        \frac{\wexc(\pi)}{2} &= -\frac{\wexc(T(\pi))}{2} + \frac{n+1}{2} \\
        \wexc(\pi) &= \frac{\wexc(\pi)}{2}-\frac{\wexc(T(\pi))}{2} + \frac{n+1}{2}. 
    \end{align*}
    This satisfies Lemma \ref{lem.telescope_strategy} with $f(\pi) = \frac{1}{2}\wexc(\pi)$.
\end{corollary}

\begin{table}[H]
    \centering
    \renewcommand{\arraystretch}{1.25}
    \hspace{-1pc}
    \begin{tabular}{@{}c l l l c l@{}}
        \toprule
        \# & Foatic map & Statistic & Associated $f(\pi)$ & Mesicity & Due to \\
        \midrule
        1 & $\C-\I$ & $\Fix(\pi)$ & $-\exc(\pi)$ & $1$ & \ref{thm.mainFix} \\
        2 & $\C-\Qsq$ & $\Fix(\pi)$ & $-\exc(\pi)$& $1$ & \ref{thm.mainFix} \\
        3 & $\R-\Qsq$ & $\Fix(\pi)$ & $-\exc(\pi)$& $1$ & \ref{thm.mainFix} \\
        4 & $\R-\I$ & $\Fix(\pi)$ & $-\exc(\pi)$& $1$ & \ref{thm.mainFix} \\
        5 & $\Qsq-\C$ & $\RtLmax(\pi) - \RtLmin(\pi)$ & $\RtLmax(\pi)$ & $0$ & \ref{thm.mainRtL} \\
        6 & $\Qsq-\R$ & $D_1(\pi)-D_{n-1}(\pi)$ & $D_1(\pi) - D_1(\I(\pi))$ & $0$ & \ref{thm.mainD} \\
        7 & $\R-\D$ & $\Fix_1(\pi) - \Fix_n(\pi)$ & $-[\pi(1) = n] - \Fix_n(\pi)$ & $0$ & \ref{thm.mainFixEnds} \\
        8 & $\D-\R$ & $\Fix_1(\pi) - \Fix_n(\pi)$ & $-[\pi(1) = n] - \Fix_n(\pi)$ & $0$ & \ref{thm.mainFixEnds} \\
        9 & $\C-\D$ & $\Fix_1(\pi) - \Fix_n(\pi)$ & $\Fix_1(\pi) + [\pi(1) = n]$ & $0$ & \ref{thm.mainFixEnds} \\
        10 & $\D-\C$ & $\Fix_1(\pi) - \Fix_n(\pi)$ & $\Fix_1(\pi) + [\pi(1) = n]$ & $0$ & \ref{thm.mainFixEnds} \\
        11 & $\C-\D$ & $\wexc(\pi)$ & $\frac{1}{2}\wexc(\pi)$ & $\frac{n+1}{2}$ & \ref{cor.wexc_sr}, \cite{Sheridan-Rossi2020} \\
        12 & $\R-\D$ & $\wexc(\pi)$ & $\frac{1}{2}\wexc(\pi)$ & $\frac{n+1}{2}$ & \ref{cor.wexc_sr}, \cite{Sheridan-Rossi2020} \\
        \bottomrule
    \end{tabular}
    \caption{Recall that $\C$ is complement, $\I$ is inversion, $\Qsq$ is rotation, $\R$ is reversal, and $\D$ is rotated-inverse.}
    \label{tab.proved_telescopes}
\end{table}

\section{\label{sect.open}Further directions}

Since the homomesies for other Foatic maps were proven previously by \cite{Sheridan-Rossi2020, LaCroixRoby2020FoaticAO} using other methods, it remains open as to what their associated telescoping functions are and whether these actions are compatible with any others. That these functions exist, and are unique up to an additive constant per-orbit, is known due to Lemma \ref{lem.telescope_strategy} and the subsequent Remark \ref{rmk.telescope_fn}, but calculating that a function fits up to some $n$ for a given statistic $s$ and action $\tau$ which is known to be $c$-mesic does not imply that it fits in general.

Along with these conjectured $f(\pi)$, we offer one correction to Theorem 2.2.11 of \cite{Sheridan-Rossi2020} where it is stated without proof that $\wexc + \exc$ is $n$-mesic for the map $\I - \I$. This is true for $n \leq 8$, but fails for $n = 9$ and we provide one counterexample orbit:

\begin{table}[H]
    \centering
    \renewcommand{\arraystretch}{1.25}
    \begin{tabular}{@{}c l l c@{}}
        \toprule
        $k$ & $\tau^k(\pi)$ (OLN) & CCD of $\tau^k(\pi)$ & $\wexc + \exc$ \\
        \midrule
        $0$ & $258673491$ & $(9,1,2,5,7,4,6,3,8)$ & $12$ \\
        $1$ & $923678451$ & $(2)(3)(8,5,7,4,6)(9,1)$ & $10$ \\
        $2$ & $915672438$ & $(9,8,3,5,7,4,6,2,1)$ & $8$ \\
        \midrule
            &             & Average & $10$ \\
        \bottomrule
    \end{tabular}
    \caption{An orbit of size $3$ of inversion-inversion $\tau = \I \circ \F^{-1} \circ \I \circ \F$ on $S_9$, with $\tau^3(\pi) = \pi$. The average of $\wexc + \exc$ over this orbit is $10 \neq 9 = n$, so $\wexc + \exc$ is not $n$-mesic under inversion-inversion.}
    \label{tab.II_counterexample}
\end{table}

Another seemingly promising but eventually false homomesy is for $\LtRmax - \RtLmin$ being $0$-mesic under the action of $\I-\I$. Again, the first counterexample only appears at $n = 9$:

\begin{table}[H]
    \centering
    \renewcommand{\arraystretch}{1.25}
    \begin{tabular}{@{}c l l c c c@{}}
        \toprule
        $k$ & $\tau^k(\pi)$ (OLN) & CCD of $\tau^k(\pi)$ & $\LtRmax$ & $\RtLmin$ & $\LtRmax - \RtLmin$ \\
        \midrule
        $0$ & $246738591$ & $(9,1,2,4,7,5,3,6,8)$ & $6$ & $1$ & $5$ \\
        $1$ & $923784651$ & $(2)(3)(7,6,4)(8,5)(9,1)$ & $1$ & $1$ & $0$ \\
        $2$ & $914723568$ & $(9,8,6,3,4,7,5,2,1)$ & $1$ & $6$ & $-5$ \\
        $3$ & $267843591$ & $(9,1,2,6,3,7,5,4,8)$ & $5$ & $1$ & $4$ \\
        $4$ & $923754861$ & $(2)(3)(5)(8,6,4,7)(9,1)$ & $1$ & $1$ & $0$ \\
        $5$ & $916732548$ & $(9,8,4,7,5,3,6,2,1)$ & $1$ & $4$ & $-3$ \\
        \midrule
            &             &         &     &   Average & $\frac{1}{6}$ \\
        \bottomrule
    \end{tabular}
    \caption{An orbit of size $6$ of inversion-inversion $\tau = \I \circ \F^{-1} \circ \I \circ \F$ on $S_9$, with $\tau^6(\pi) = \pi$. The average of $\LtRmax - \RtLmin$ over this orbit is $\frac{1}{6} \neq 0$, so $\LtRmax - \RtLmin$ is not $0$-mesic under inversion-inversion, although it is $0$-mesic for all $n \leq 8$.}
    \label{tab.II_counterexample_LtR}
\end{table}

We do however conjecture the following list of telescoping functions corresponding to each given statistic and action pair.

\begin{table}[H]
    \centering
    \renewcommand{\arraystretch}{1.25}
    \begin{tabular}{@{}c l l l c@{}}
        \toprule
        \# & Foatic map & Statistic & Conjectured $f(\pi)$ & Mesicity\\
        \midrule
        1 & $\C-\C$ & $\Fix_1(\pi) - \Fix_n(\pi)$ & $\Fix_1(\pi)$ & $0$ \\
        2 & $\R-\R$ & $\Fix_1(\pi) - \Fix_n(\pi)$ & $-\Fix_n(\pi)$ & $0$\\
        3 & $\Qsq-\C$ & $\Fix_1(\pi) - \Fix_n(\pi)$ & $-\Fix_n(\pi) - [\pi(n) = 1]$ & $0$\\
        4 & $\Qsq-\R$ & $\Fix_1(\pi) - \Fix_n(\pi)$ & $\Fix_1(\pi) + [\pi(n) = 1]$ & $0$\\
        5 & $\I-\Qsq$ & $\Fix_1(\pi) - \Fix_n(\pi)$ & $\Fix_1(\pi)$ & $0$\\
        6 & $\I-\D$ & $\Fix_1(\pi) - \Fix_n(\pi)$ & $\Fix_1(\pi)$ & $0$\\
        7 & $\D-\I$ & $\Fix_1(\pi) - \Fix_n(\pi)$ & $\Fix_1(\pi)$ & $0$\\
        8 & $\C-\Qsq$ & $\Fix_1(\pi) - \Fix_n(\pi)$ & $-\Fix_n(\pi) - [\cyc(\pi) = 1]$ & $0$\\
        9 & $\R-\Qsq$ & $\Fix_1(\pi) - \Fix_n(\pi)$ & $\Fix_1(\pi) + [\cyc(\pi) = 1]$ & $0$\\
        10 & $\Qsq-\I$ & $\wexc(\pi) + \exc(\pi)$ & $\exc(\pi)$ & $n$ \\
        11 & $\Qsq-\Qsq$ & $\wexc(\pi) + \exc(\pi)$ & $\exc(\pi)$ & $n$ \\
        12 & $\Qsq-\I$ & $\LtRmax(\pi) - \RtLmin(\pi)$ & $\LtRmax(\pi)$ & $0$ \\
        13 & $\Qsq-\Qsq$ & $\LtRmax(\pi) - \RtLmin(\pi)$ & $\LtRmax(\pi)$ & $0$ \\
        \bottomrule
    \end{tabular}
    \caption{The telescoping identity $s(\pi) = f(\pi) - f(\tau(\pi)) + c$ for each conjecture has been verified exactly at $n = 4,5,6,7$. For rows 5--7, the function $-\Fix_n(\pi)$ also seems to serve as a telescoping function (equivalently, the two functions differ by a per-orbit constant in the sense of Remark~\ref{rmk.telescope_fn}).}
    \label{tab.conjectured_telescopes}
\end{table}

%\nocite{*}
\bibliographystyle{amsalpha}
\bibliography{Refs}

@misc{LaCroixRoby2020FoaticAO,
  author        = {{La Croix}, Michael and Roby, Tom},
  title         = {{F}oatic actions of the symmetric group and fixed-point homomesy},
  howpublished  = {Preprint, arXiv:2008.03292 [math.CO]},
  year          = {2020},
  eprint        = {2008.03292},
  archivePrefix = {arXiv},
  primaryClass  = {math.CO},
  url           = {https://arxiv.org/abs/2008.03292},
}

@phdthesis{Sheridan-Rossi2020,
  author  = {{Sheridan-Rossi}, Elizabeth},
  title   = {Homomesy for {F}oatic Actions on the Symmetric Group},
  school  = {University of Connecticut},
  address = {Storrs, CT},
  year    = {2020},
  url     = {https://digitalcommons.lib.uconn.edu/dissertations/2620},
}

@article{Propp_Roby_2015,
  author  = {Propp, James and Roby, Tom},
  title   = {Homomesy in products of two chains},
  journal = {Electronic Journal of Combinatorics},
  volume  = {22},
  number  = {3},
  pages   = {Paper 3.4},
  year    = {2015},
  doi     = {10.37236/3579},
  url     = {https://doi.org/10.37236/3579},
}

@incollection{Roby2016,
  author    = {Roby, Tom},
  title     = {Dynamical algebraic combinatorics and the homomesy phenomenon},
  booktitle = {Recent Trends in Combinatorics},
  editor    = {Beveridge, Andrew and Griggs, Jerrold R. and Hogben, Leslie and Musiker, Gregg and Tetali, Prasad},
  series    = {The IMA Volumes in Mathematics and its Applications},
  volume    = {159},
  publisher = {Springer},
  address   = {Cham},
  year      = {2016},
  pages     = {619--652},
  doi       = {10.1007/978-3-319-24298-9_25},
  url       = {https://doi.org/10.1007/978-3-319-24298-9_25},
}

@article{ElderEtAl2024,
  author        = {Elder, Jennifer and Lafreni{\`e}re, Nadia and McNicholas, Erin and Striker, Jessica and Welch, Amanda},
  title         = {Homomesies on permutations: {A}n analysis of maps and statistics in the {FindStat} database},
  journal       = {Mathematics of Computation},
  volume        = {93},
  number        = {346},
  pages         = {921--976},
  year          = {2024},
  doi           = {10.1090/mcom/3866},
  url           = {https://doi.org/10.1090/mcom/3866},
  eprint        = {2206.13409},
  archivePrefix = {arXiv},
}

@book{FoataSchutzenberger1970,
  author    = {Foata, Dominique and Sch{\"u}tzenberger, Marcel-Paul},
  title     = {Th{\'e}orie g{\'e}om{\'e}trique des polyn{\^o}mes {E}ul{\'e}riens},
  series    = {Lecture Notes in Mathematics},
  volume    = {138},
  publisher = {Springer-Verlag},
  address   = {Berlin},
  year      = {1970},
  doi       = {10.1007/BFb0060799},
  url       = {https://doi.org/10.1007/BFb0060799},
}

@book{Bona-CP3,
  author    = {B{\'o}na, Mikl{\'o}s},
  title     = {Combinatorics of Permutations},
  edition   = {3rd},
  series    = {Discrete Mathematics and its Applications},
  publisher = {CRC Press},
  address   = {Boca Raton, FL},
  year      = {2022},
  doi       = {10.1201/9780429274107},
  url       = {https://doi.org/10.1201/9780429274107},
}

@article{Renyi1962,
  author  = {R{\'e}nyi, Alfr{\'e}d},
  title   = {Th{\'e}orie des {\'e}l{\'e}ments saillants d'une suite d'observations},
  journal = {Annales de la Facult{\'e} des Sciences de l'Universit{\'e} de Clermont. Math{\'e}matiques},
  volume  = {8},
  number  = {2},
  pages   = {7--13},
  year    = {1962},
  url     = {https://www.numdam.org/item/ASCFM_1962__8_2_7_0/},
}

@book{EC1,
  author    = {Stanley, Richard P.},
  title     = {Enumerative Combinatorics, Volume 1},
  edition   = {2nd},
  series    = {Cambridge Studies in Advanced Mathematics},
  volume    = {49},
  publisher = {Cambridge University Press},
  address   = {Cambridge},
  year      = {2012},
  doi       = {10.1017/CBO9781139058520},
  url       = {https://doi.org/10.1017/CBO9781139058520},
}

@article{DefantHopkinsPoznanovicPropp2023,
  author        = {Defant, Colin and Hopkins, Sam and Poznanovi{\'c}, Svetlana and Propp, James},
  title         = {Homomesy via toggleability statistics},
  journal       = {Combinatorial Theory},
  volume        = {3},
  number        = {2},
  pages         = {Paper No. 14},
  year          = {2023},
  doi           = {10.5070/C63261992},
  url           = {https://doi.org/10.5070/C63261992},
  eprint        = {2108.13227},
  archivePrefix = {arXiv},
}

@misc{Propp2021spectral,
  author        = {Propp, James},
  title         = {A spectral theory for combinatorial dynamics},
  howpublished  = {Preprint, arXiv:2105.11568 [math.CO]},
  year          = {2021},
  eprint        = {2105.11568},
  archivePrefix = {arXiv},
  url           = {https://arxiv.org/abs/2105.11568},
}

@article{Dowling_2025,
  author        = {Dowling, Will and Lafreni{\`e}re, Nadia},
  title         = {Homomesy on permutations with toggling actions},
  journal       = {Involve},
  volume        = {18},
  number        = {5},
  pages         = {829--854},
  year          = {2025},
  doi           = {10.2140/involve.2025.18.829},
  url           = {https://doi.org/10.2140/involve.2025.18.829},
  eprint        = {2312.02383},
  archivePrefix = {arXiv},
}

\end{document}